\documentclass{article}[11pt]
\usepackage{amsmath,authblk,graphicx,verbatim,hyperref}
\usepackage{amssymb}
\usepackage{amsthm}

\usepackage[margin=1.20in]{geometry}

\newtheorem{theorem}{Theorem}
\newtheorem{lemma}{Lemma}

\def \D{\mathbb{D}}
\def \H{\mathbb{H}}
\def \E{\mathbb{E}}
\def \P{\mathbb{P}}
\def \C{\mathbb{C}}
\def \Z{\mathbb{Z}}
\def \N{\mathbb{N}}
\def\I {\mathbb{I}}
\def \R{\mathbb{R}}

\def \({\left(}
\def \){\right)}
\def \[{\left[}
\def \]{\right]}
\def\ha{\widehat}
\def\lin{\overline}
\def\sem{\setminus}
\def\pa{\partial}
\def\til{\widetilde}
\def\no{\noindent}

  \DeclareMathOperator{\Cont}{Cont}
 \DeclareMathOperator{\Area}{Area}
\DeclareMathOperator{\crad}{crad} \DeclareMathOperator{\doub}{doub}
 
\DeclareMathOperator{\dist}{dist} 
 
\DeclareMathOperator{\Imm}{Im } \DeclareMathOperator{\Ree}{Re }

\DeclareMathOperator{\rad}{rad}

\begin{document}
\title{Multipoint Estimates for Radial and Whole-plane SLE}
\author{Benjamin Mackey\thanks{Michigan State University and Cleveland State University.} and Dapeng Zhan\thanks{Michigan State University. Partially supported by NSF grant  DMS-1056840 and Simons Foundation grant \#396973.}}
\maketitle

\begin{abstract}
We prove upper bounds for the probability that a radial SLE$_{\kappa}$ curve, $\kappa\in(0,8)$, comes within specified radii of $n$ different points in the unit disc. Using this estimate, we then prove a similar upper bound for  a whole-plane SLE$_{\kappa}$ curve. We then use these estimates to show that the lower Minkowski content of both the radial and whole-plane SLE$_{\kappa}$ traces restricted in a bounded region have finite moments of any order.
\end{abstract}

\section{Introduction}
The goal of this paper is to prove estimates for the probability that radial SLE and whole-plane SLE paths pass near any finite collection of points. These estimates are then used to show that an important geometric quantity of the path has finite moments. The Schramm-Loewner evolution, abbreviated as SLE, is a family of random processes first introduced by Oded Schramm as a candidate for the continuous scaling limit of several discrete lattice models from statistical physics \cite{Schramm}. The process SLE$_{\kappa}$ depends on $\kappa>0$, and for various specific values of $\kappa$, such as $2,3,4,16/3,6,8$, SLE$_{\kappa}$ has been proven to be the scaling limit of some lattice model (\cite{LSW,Smirnov perc,Smirnov Ising,Schramm Sheff}).

The above results about scaling limits all assume that the involved curves have a particular parametrization. In particular, they assume that the capacity of the curves grows at a constant rate. This capacity parametrization is convenient for calculations about the SLE curves, but is not natural for the lattice models. We want to be able to run the paths in the lattice model so that each segment takes the same length of time. The first instinct may be to parametrize the SLE path by arc length and check for convergence, but Beffara \cite{Beffara} proved that the dimension of the SLE$_{\kappa}$ curve is $d:=(1+\kappa/8) \wedge 2$, and so the arc length is always infinite. If $\kappa \geq 8$, the SLE paths are space filling \cite{Rhode Schramm}. We will focus in this paper on the case $\kappa<8$.

There has been recent work developing a $d$-dimensional measurement of length which can be used to parametrize the SLE paths, which has been called the natural parametrization. In \cite{Law Sheff}, the Doob-Meyer theorem was used to create an increasing process which was called the natural parametrization, or natural length, for $\kappa<5.021...$, and it was conjectured to coincide with the $d$-dimensional Minkowski content of the curve. The $d$-dimensional Minkowski content of a set $E \subset \C$ is defined by
$$\Cont_d(E)=\lim_{r \to 0} r^{d-2} \Area\{ z \in \C: \dist(z,E) < r \},$$ provided that the limit exists. The  lower Minkowski content $\underline{\Cont}_d(E)$ is similarly defined with the limit replaced by the lower limit, which always exists.
In \cite{Law Zhou}, the Doob-Meyer construction was extended to all $\kappa<8$. In \cite{Law Rez mink content}, it was proven that the $d$-dimensional Minkowski content exists almost surely, and that it agrees with the natural parametrization already constructed. In \cite{Law Vik}, it was proven that appropriately time scaled loop erased random walk converges to chordal SLE$_2$ in the natural parametrization. Later the result was extended to radial SLE$_2$ (\cite{Law Vik radial}).

An important tool in the construction and analysis of the natural length is the Green's function, which gives the normalized probability that the SLE path passes through a point. Speaking more precisely, the Green's function at $z$ is defined by $G(z)=\lim_{r \downarrow 0}r^{d-2}\P[\dist(\gamma,z) \leq r]$, where $\gamma$ is the SLE path, assuming this limit exists. 
More generally, the multipoint Green's function gives the normalized probability that the path passes through multiple fixed points. For distinct points $z_1,\dots, z_n$, the multipoint Green's function is defined by
\begin{equation}\label{eq: multipoint Green's function}
G(z_1, \dots, z_n)=\lim_{r_1,\dots,r_n \downarrow 0} \(\prod_{k=1}^n r_k^{d-2}\) \P\[  \bigcap_{k=1}^n \{\dist(\gamma, z_k) < r_k \}\],
\end{equation}
if this limit exists.

There are several varieties of SLE, including chordal SLE which connects boundary points, radial SLE which connects a boundary point and an interior point, and whole-plane SLE which connects interior points. Each type of SLE has its own Green's functions. The
one-point chordal SLE Green's function $G$ with conformal radius in place of Euclidean distance was first shown to exist and used in \cite{Rhode Schramm}, and the exact form is known. 
In \cite{Law Werness}, the two point Green's function in terms of conformal radius was proven to exist. In \cite{Law Rez mink content}, the one-point and two-point chordal Green's functions in terms of Euclidean distance, i.e., the original definition, are proven to exist, and differ from the conformal radius version of Green's functions in \cite{Rhode Schramm,Law Werness} by some multiplicative constants depending only on $\kappa$.

In \cite{Law Werness}, the authors conjectured that their construction for the two point Green's function can be generalized to show that the higher order Green's function exists. For any $n \geq 1$, the authors of \cite{Rez Zhan upper bound} find an upper bound for the probability that chordal SLE passes near $n$ points, which extends the upper bound in \cite{Law Rez two point Green's} for $n=2$. Using this upper bound, they also prove that the Minkowski content of the chordal SLE path has finite $n$-th moment for any $n>0$. In \cite{Rez Zhan lower bound}, the same authors prove that the bound in \cite{Rez Zhan upper bound} is sharp up to a multiplicative constant, and using this sharp bound they prove that the limit (\ref{eq: multipoint Green's function}) for chordal SLE exists for all distinct points $z_1, \dots, z_n$ in the upper half-plane. They also find some rate of convergence and modulus of continuity for the Green's functions.

The Green's functions for radial SLE are less well studied. Existence of the
one-point conformal radius Green's function is proven in \cite{AKL}, but an exact form is only found for $\kappa=4$. In this paper, we use one-point estimates found in \cite{AKL} and follow the strategy in \cite{Rez Zhan upper bound} to prove the following multipoint estimate:

\begin{theorem}\label{thm: multi point estimate radial}Fix $\kappa \in (0,8)$. Let $\gamma$ be a radial SLE$_{\kappa}$ in the unit disc $\D$ from $1$ to $0$, let $z_1,\dots,z_n$ be distinct points in $\overline{\D}\sem\{1,0\}$, and let $z_0=1$. Let $y_k=1-|z_k|$ be the distance of each point to the boundary of $\D$, and define $l_k=\min\{ |z_k|,|z_k-1|, |z_k-z_1|, \dots, |z_k-z_{k-1}|  \}$. Then there exists an absolute constant $C_n\in(0,\infty)$ depending only on $\kappa$ and $n$ such that
\begin{equation}\label{eq: multipoint bound radial}
\P[ \bigcap_{k=1}^n \{\dist(\gamma, z_k) < r_k \}] \leq C_n \prod_{k=1}^{n} \frac{P_{y_k}(r_k \wedge l_k)}{P_{y_k(l_k)}}.
\end{equation}
\end{theorem}

The functions $P_y(x)$ will be defined in the next section, but the idea is that these functions represent both interior and boundary estimates. whole-plane SLE is closely related to radial SLE, and we will show that Theorem \ref{thm: multi point estimate radial} implies a similar estimate for the whole-plane SLE trace.

\begin{theorem}\label{thm: multi point estimate whole-plane}
Fix $ \kappa \in (0,8)$, and let $\gamma^*$ be a whole-plane SLE$_{\kappa}$ trace from  $\infty$ to $0$. Let $z_1, \dots, z_n$ be distinct points in $\C \backslash \{0\}$. For each $k=1, \dots, n,$ let $0 < r_k <|z_k|$ and define $l_k=\min\{|z_k|, |z_k-z_1|, \dots, |z_k-z_{k-1}|\}$. Then there is a constant $C_n\in(0,\infty)$ depending only on $\kappa$ and $n$ such that
\begin{equation}\label{eq: mulipoint bound whole-plane}
\P[ \bigcap_{k=1}^n \{ \dist(\gamma^*,z_k) < r_k \} ] \leq C_n \prod_{k=1}^n \( \frac{r_k\wedge l_k}{l_k} \)^{2-d}.
\end{equation}
\end{theorem}
Note that the expression of this bound is simpler than Theorem \ref{thm: multi point estimate radial}, since there are no boundary effects with which to be concerned.

We want to show that the Minkowski content for radial SLE has all finite moments, like was done in the chordal case in \cite{Rez Zhan upper bound}, but the Minkowski content of the radial SLE path has not been rigorously proven to exist. 
Using the the existence of Minkowski content of chordal SLE and the weak equivalence between radial SLE and chordal SLE, one can easily prove the existence of the Minkowski content of a radial SLE curve up to any time that the curve does not reach its target and the boundary of the domain is not completely separated from the target. The existence has not been extended to the whole radial SLE curve (including its end point). In the theorem below, we work on the lower Minkowski content instead to avoid this issue. When $\kappa\in(0,4]$, the radial SLE$_\kappa$ curve has Minkowski content in any domain that does not contain a neighborhood of the target, so the theorem implies upper bound of moments of Minkowski contents of the curve restricted in such domains.


\begin{theorem}\label{thm: Minkowski content moments}
Fix $\kappa \in (0,8)$.
\begin{itemize}
\item[a)] Let $\gamma$ be a radial SLE$_{\kappa}$ trace in $\D$ from $1$ to $0$. Then $\E[ \underline{\Cont}_d(\gamma)^n ]< \infty$ for all $n \in \N$.
\item[b)] Let $\gamma^*$ be a whole-plane SLE$_{\kappa}$ trace from $0$ to $\infty$, and suppose $D \subset \C$ is compact. Then $\E[\underline{\Cont}_d(\gamma^* \cap D)^n]<\infty$ for every $n \in \N$.
\end{itemize}
\end{theorem}

The paper will be organized as follows. First we review preliminary information, which will be used in this paper. 
Next, we provide one-point estimates for radial SLE in the forms which will be useful for us. We then use these one-point estimates to prove some key lemmas, followed by the proofs of the main theorems.

\section{Preliminaries}

\subsection{General notation}
Throughout, we fix $\kappa\in(0,8)$.
A constant is a positive finite number that usually depends only on $\kappa$, and we often denote it by $C$. Sometimes it may also depend on some other parameter such as an integer $n$, in which case we use  $C_n$. We write $X\lesssim Y$ or $Y\gtrsim X$ if there is a  constant $C>0$ such that $X\le CY$. We write $X\asymp Y$ if $X\lesssim Y$ and $X\gtrsim Y$. 

 The function $P_y(x)$ used in Theorem \ref{thm: multi point estimate radial} is defined by
$$P_y(x)= \begin{cases}
y^{\alpha-(2-d)} x^{2-d}, & x \leq y\\
x^{\alpha}, & x \geq y
\end{cases},$$ where $d=1+\kappa/8$ is the Hausdorff dimension of the SLE$_\kappa$ path, and $\alpha=8/\kappa-1$ is related to the boundary exponent for SLE$_{\kappa}$. This upper bound mixes the estimates for interior points and points near the boundary. Roughly speaking, if the point $z_k$ is far from the boundary, the term on the right hand side of (\ref{eq: multipoint bound radial}) corresponding to $z_k$ is close to $(r_k/l_k)^{2-d}$. If $z_k$ is near the boundary of the unit disc, then the corresponding term on the right side of (\ref{eq: multipoint bound radial}) is close to $(r_k/l_k)^{\alpha}$. If $z_k$ is neither too close or too far away from the boundary, then the corresponding estimate is a mixture of the two.

The following Lemma about the functions $P_y$ is Lemma 2.1 in \cite{Rez Zhan lower bound}, and can be proven with a case by case argument.
\begin{lemma}\label{lem: P_y bounds}
For $0\leq x_1 <x_2,$ $0 \leq y_1 \leq y_2,$ $0<x,$ and $0\leq y$, we have
$$\frac{P_{y_1}(x_1)}{P_{y_1}(x_2)} \leq \frac{P_{y_2}(x_1)}{P_{y_2}(x_2)};$$
$$\(\frac{x_1}{x_2}\)^{\alpha} \leq \frac{P_{y}(x_1)}{P_y(x_2)} \leq \( \frac{x_1}{x_2} \)^{d-2} =\frac{P_{x_2}(x_1)}{P_{x_2}(x_2)};$$
$$\( \frac{y_1}{y_2} \)^{\alpha-(2-d)} \leq \frac{P_{y_1}(x)}{P_{y_2}(x)} \leq 1.$$
\end{lemma}


\subsection{Radial SLE}
There are several varieties of the Loewner equation. We will focus on the radial Loewner equation, and will also need the covering radial Loewner equation. Complete details can be found in \cite{Law}.
Let $\D=\{ z \in \C: |z|<1\}$ be the unit disc. A $\D$-hull is a set $K \subset D$ which is relatively closed in $\D$, $0 \notin K$, and $\D \backslash K$ is simply connected. By the Riemann mapping theorem, there is a unique conformal map $g_K: \D \backslash K \to \D$ with $g_K(0)=0$ and $g_K'(0)>0$. The capacity of $K$ is defined by cap$(K)=\ln(g_K'(0))$. We will call a subset $D \subset \D$ a $\D$-domain if $D=\D \backslash K$ for a $\D$-hull $K$.

Given any real valued continuous function $\lambda\in C[0,\infty)$, the radial Loewner equation driven by $\lambda$ is, for all $z \in \D$,
\begin{equation}\label{eq: radial Loewner eq}
\partial_t g_t(z) = g_t(z) \frac{e^{i\lambda(t)} +g_t(z)}{e^{i\lambda(t)}-g_t(z)}, \text{ and } g_0(z)=z.
\end{equation}
If $\tau_z$ is the lifetime of (\ref{eq: radial Loewner eq}) at $z \in \D$ and $K_t=\{ z \in \D : \tau_z \leq t  \}$, then $K_t$ is a $\D$-hull with cap$(K_t)=t$ and $g_t$ is the conformal map $g_{K_t}$ associated with $K_t$.

For $\kappa>0$, the radial SLE$_{\kappa}$ process is the solution (\ref{eq: radial Loewner eq}) for $\lambda(t)=\sqrt{\kappa}B_t$, where $B_t$ is standard one dimensional Brownian motion. Similarly to the chordal case, there is a radial trace $\gamma: [0,\infty) \to \overline{\D}$ so that $\D \backslash K_t$ is the component of $\D \backslash \gamma[0,t]$ containing $0$ with $\gamma(0)=1$ and $\gamma(\infty)=0$. The radial SLE$_{\kappa}$ trace has the same phrase transitions and the same dimension as the chordal SLE$_{\kappa}$. 


The radial SLE described above is the standard radial SLE$_\kappa$ in $\D$ from $1$ to $0$.
Given any simply connected domain $D$,
a prime end (\cite{Ahlf}) $a$ of $D$, and interior point $b \in D$, SLE$_{\kappa}$ in $D$ from $a$ to $b$ is obtained by applying a conformal map $\phi: \D \to D$ with $\phi(1)=a$ and $\phi(0)=b$ to the standard radial SLE$_\kappa$ curve, so that the curve in $D$ grows from a boundary point to an interior point. 

If $\gamma$ is a radial SLE$_{\kappa}$ in a domain $D$ from $a$ to $b$, and $T$ is any  stopping time for $\gamma$ at which $\gamma$ does not reach $b$, the domain Markov property (DMP) says that, conditioned on $\gamma[0,T]$, $\gamma^T(t):=\gamma(T+t)$, $t\ge 0$, is a radial SLE$_{\kappa}$ path in a complement domain of $\gamma[0,T]$ in $D$ from $\gamma(T)$ to $b$.

\subsection{Radial SLE in the cylinder}
Let $\H=\{z\in\C:\Imm z>0\}$ be the upper half plane. It can be seen as a covering space for the unit disc $\D\sem\{0\}$ under the map $e^{2i}:z\mapsto e^{2i z}$. Let $\H^*$ be the cylinder defined by declaring that $z,w\in \H$ are equal if $z-w \in \pi \Z$. Then $e^{2i}$ induces a conformal map, still denoted by $e^{2i}$, from $\H^*$ onto $\D\sem\{0\}$. The boundary of $\H^*$ is $\R^*$, which is equal to $\R$ modulo the same equivalence relation.  By  defining $e^{2i}(\infty)=0$, we extend $e^{2i}$ to a conformal map from $\H^*\cup\{\infty\}$ onto $\D$. Thus, the image of a standard radial SLE$_\kappa$ under $(e^{2i})^{-1}$ is a radial SLE$_\kappa$ in $\H^*$ from $0$ to $\infty$.

For $z',w' \in \overline{\H}^*$ which can be represented as $z+\pi \Z, w+\pi \Z$ respectively for $z,w \in \overline{\H}$, the distance from $z'$ to $w'$ in $\overline{\H}^*$ is defined to be the Euclidean distance between the sets $z+\pi\Z$ and $w+\pi\Z$ in $\overline{\H}$. It will be written as $|z'-w'|_*$ to distinguish from the distance between points in $\C$. 
 Similarly, if $A',B'\subset \overline{\H}^*$ with $A'=A+\pi\Z$ and $B'=B+\pi\Z$, then the distance from $A'$ to $B'$ in $\H^*$ is the Euclidean distance from $A+\pi\Z$ to $B+\pi\Z$, and is denoted by $\dist_{\H^*}(A',B')$. Given any $z' \in \overline{\H}^*$ represented by $z+\pi \Z$, the ball of radius $r$ centered at $z'$ is denoted by $B_{\H^*}(z',r)$, and is represented in $\overline{\H}$ by $(B(z,r)\cap \H)+\pi\Z$. Note that for $r<\pi/2$, the representatives of $B_{\H^*}(z',r)\cap\H$ are nonoverlapping.

We call $K\subset \H^*$ an $\H^*$-hull if $e^{2i}(K)$ is a $\D$-hull. The complement of an $\H^*$-hull in $\H^*$ is called an $\H^*$-domain. Recall that for a simply connected domain $D\subsetneqq \C$ and an interior point $z\in D$, the conformal radius of $D$ seen from $z$ is  $\crad_D(z):=|\phi'(0)|$, if $\phi$ is a conformal map from $\D$ onto $D$ with $\phi(0)=z$.
By adding $\infty$ to $\H^*$ to make it simply connected, we may use the same spirit to define $\crad_{\H^*}(z')$ for any $z'\in \H^*$, and obtain $\crad_{\H^*}(z')=|\phi'(0)|$, where $\phi:=(e^{2i})^{-1}\circ \psi$ and $\psi$ is a M\"obius automorphism of $\D$ that sends $0$ to $e^{2i z'}$.  It is easy to calculate that
$\crad_{\H^*}(z')=\sinh(2\Imm z')$.
We may similarly define the conformal radius of an $\H^*$-domain. If $D$ is an $\H^*$-domain and $z'\in D$, then there is a conformal map $g$ from $D$ onto $\H^*$ such that $\Imm g(z)\to \infty$ as $\Imm z\to \infty$. It is easy to calculate that $\crad_D(z')=\sinh(2\Imm g(z'))/|g'(z')|$.

Koebe's $1/4$ theorem states that, for a simply connected domain $D\subsetneqq\C$, the conformal radius is comparable to the in-radius. More precisely, $\dist(z,\pa D)\le \crad_D(z)\le 4\dist(z,\pa D)$ for any $z\in D$. However, Koebe's $1/4$ theorem does not hold for $\H^*$ or $\H^*$-domains. In fact, $\dist(z',\pa \H^*)=\Imm z'$ is not comparable to $\crad_{\H^*}(z')=\sinh(2\Imm z')$. However, we may still apply Koebe's $1/4$ theorem to any simply connected subdomain of $\H^*$ (not containing $\infty$). See the lemma below.

\begin{lemma}
Let $D$ be a $\D$-domain, and let $H$ be an $\H^*$-domain with $e^{2i}(H)=D$. Let $z_0 \in D$ and $w_0'\in H$ so that $e^{2iw_0'}=z_0$. If $y_0=1-|z_0| \leq 1/2$, then $\dist(z_0,\partial D) \leq 4\dist_{\H^*}(w_0',\partial H)$. \label{lem: Koebe}
\end{lemma}
\begin{proof}
Let $\phi=(e^{2i})^{-1}:D\to H$. Let $r=\dist(z_0,\partial D)$. Then $r\le y_0$. The assumption $y_0 \leq 1/2$ implies that $\phi$ restricted to $B(z_0,r)$ may be lifted to a conformal map $\psi$ from $D$ into $\H$. Let $w_0=\psi(z_0)\in\H$. Then $w_0'$ is represented by $w_0+\pi\Z$.  By Koebe's $1/4$ theorem, $B(w_0, r/4) \subset \psi(B(z_0,r))$. So $B_{\H^*}(w_0',r/4)\subset \phi(B(z_0,r))\subset H$, and the conclusion easily follows.
\end{proof}

\subsection{Crosscuts and prime ends}
In later sections, we will be studying the behavior of the radial SLE curve as it crosses many interior curves, creating different components of the initial domain $\D$. We need to introduce some notation which will make it easier to distinguish which component is discussed at any point in time. This is the same framework introduced in \cite{Rez Zhan upper bound}.

Recall that a crosscut in a domain $D \subset \C$ is a simple curve $\rho:(a,b) \to D$ such that $\lim_{t \to a^+}\rho(t):=\rho(a^+)$ and $\lim_{t \to b^-}\rho(t):=\rho(b^-)$ both exist and are elements of the boundary of $D$. Then $\rho$ lies inside of $D$, but the endpoints do not. The endpoints $\rho(a^+)$ and $\rho(b^-)$ determine prime ends for the domain. If $D$ is simply connected, and $f$ maps $D$ conformally onto a Jordan domain $D'$, then $f(\rho)$ is a crosscut in $D'$. More information about crosscuts and prime ends can be found in \cite{Ahlf}.

Note that if $\rho$ is a crosscut in a simply connected domain $D$, then $\rho$ divides $D$ into two components. Even more generally, let $K \subset D$ be relatively closed. Let $S$ be either a connected subset of $D \backslash K$ or a prime end of $D\backslash K$. We then define $D(K;S)$ to be the component of $D\backslash K$ which contains $S$. We also introduce the symbol $D^*(K;S)= D \backslash \( K \cup D(K;S) \)$, which is the union of the remaining components of $D \backslash K$. This notation is useful for expressing whether $K$ separates points. For example, if $\rho \subset D$ is a crosscut which separates two points $z, w \in D$, then $D(\rho; z) \neq D(\rho;w)$. In fact, in this case, we have $D(\rho; w) = D^*(\rho;z)$, and $D(\rho;z) = D^*(\rho;w)$.

Since we will be working with domains which have $0$ as an interior point, and in particular will be concerned with components containing $0$, we will use $D(K)$ and $D^*(K)$ to represent $D(K;0)$ and $D^*(K;0)$ respectively. Note that this is a departure from the notation in \cite{Rez Zhan upper bound}, where the point being suppressed was the prime end $\infty$. The change is to reflect the fact that the target of the radial SLE curve is the interior point $0$.


The next lemma is \cite[Lemma 2.1]{Rez Zhan upper bound}:
\begin{lemma}\label{lem: first subcrosscut}
Let $D \subset \til{D}$ be simply connected domains in $\C$. Let $\rho$ either be a Jordan curve in $\til{D}$ which intersects $\partial D$ or a crosscut in $\til{D}$. Let $Z_1,Z_2$ be two connected subsets or prime ends of $\til{D}$ such that $\til{D}(\rho;Z_j)$ is well defined for both $j=1,2$, and are nonequal. This means that $\til{D}\backslash \rho$ is a neighborhood of both $Z_1$ and $Z_2$ in $\til D$, and $Z_1$ is disconnected from $Z_2$ in $\til{D}$ by $\rho$.

Suppose that $D$ is a neighborhood of $Z_1$ and $Z_2$ in $\til{D}$. Let $\Lambda$ be the set of connected components of $D \cap \rho$. Then there exists a unique $\lambda_1 \in \Lambda$ such that $D(\lambda_1;Z_1) \neq D(\lambda_1;Z_2)$, and if $\lambda \in \Lambda$ such that $D(\lambda;Z_1)\neq D(\lambda;Z_2)$, then $D(\lambda_1; Z_1) \subset D(\lambda; Z_1)$ and $D(\lambda; Z_2) \subset D(\lambda_1; Z_2) $.
\end{lemma}

The $\lambda_1$ obtained in Lemma \ref{lem: first subcrosscut} will be referred to as the first subcrosscut of $\rho$ to disconnect (or separate) $Z_1$ and $Z_2$ in $D$. The conclusion of the lemma states that of all subcrosscuts of $\rho$ in $D$ which disconnect $Z_1$ and $Z_2$, $\lambda_1$ is closest to $Z_1$ in the sense that the component containing $Z_1$ it determines is contained in the component determined by any other such subcrosscut.

\subsection{Extremal length and distortion theorem}
Let $d_{\Omega}(A,B)$ denote the extremal distance from $A$ to $B$ in $\Omega$. For the definition of extremal distance, see \cite{Ahlf}. Note that this is distinct from the notations $\dist(A,B)$ and $\dist_{\H^*}(A,B)$, both of which represent Euclidean distance. Define $\Lambda(R)=d_{\Omega_R}( [-1,0], [R,\infty) )$, where $\Omega_R= \C \backslash\{[-1,0] \cup [R,\infty)\}$. By Teichm\"uller's theorem (\cite{Ahlf}), this is maximal among doubly connected domains in modulus (extremal distance between boundary components)  which separate $\{-1,0\}$ and $\{w, \infty\}$ with $|w|=R$. Moreover, $\Lambda(R) \leq \frac{1}{2\pi} \ln( 16(R+1))$ for all values of $R \geq 1$.

Combining Teichm\"uller's theorem with the reflection principle of the extremal length (about $\R$), one easily obtains the following lemma.

\begin{lemma}\label{lem: extremal distance hull}
Let $\eta$ be a crosscut in $\{z \in \H: \Ree(z)>0\}$ with endpoints $0< a<b$. Define $r=\sup\{| z-a | : z \in \eta\}$. Then
$$\min\{1,r/a\}\lesssim e^{-\pi d_{\H}(\eta, (-\infty,0])}.$$
\end{lemma}

The following application of Koebe's distortion theorem (\cite{Ahlf}) will be used repeatedly in the next section to show that interior estimates are comparable after applying conformal maps. 

\begin{lemma}\label{lem: Growth thm est}
Let $\Omega \subsetneqq \C$ be a domain, $z_0\in D$, and  $M=\dist(z_0,\partial \Omega)$. Suppose that $\phi$ is a conformal map defined on $D$, $R\le M$,  and $0<r <R/7$. Define $$\til{R}=\frac{R |\phi'(z_0)|}{(1+R/M)^2} , \quad \til{r}=\frac{r |\phi'(z_0)|}{(1-r/M)^2}.$$
Then $\frac{r}{R} < \frac{\til{r}}{\til{R}} <7 \frac{r}{R}<1$, and
\begin{equation}
\phi(B(z_0,r)) \subset B(\phi(z_0),\til{r}) \subset B(\phi(z_0), \til{R}) \subset \phi(B(z_0,R)), \label{eq: inclusion}
\end{equation}
\end{lemma}
\begin{proof}
First, the inequalities $\frac{r}{R} < \frac{\til{r}}{\til{R}} <7 \frac{r}{R}<1$ follow easily from $R\le M$ and $0<r<R/7$.
Applying Koebe's distortion theorem to the univalent map $f: \D \to \C$ defined by $$f(w)=\frac{\phi(z_0+Mw)-\phi(z_0)}{M\phi'(z_0)},$$
we get that if $|z-z_0|= \rho \in (0,M)$, then
\begin{equation}\label{eq: Growth thm}
\frac{\rho {|\phi'(z_0)|}}{(1+\rho/M)^2} \leq  {|\phi(z)-\phi(z_0)|} \leq \frac{\rho{|\phi'(z_0)|}}{(1-\rho/M)^2}.
\end{equation}
Since $\til{r}$ is the righthand side for $\rho=r$, and $\til{R}$ is  the lefthand side for $\rho=R$, we get (\ref{eq: inclusion}).
\end{proof}

\section{Interior and boundary estimates}
We recall some estimates for radial SLE from \cite{AKL} and further develop them. 
The boundary estimate describes how difficult can a radial SLE$_\kappa$ curve gets close to a marked boundary point. The following is \cite[Lemma 5.1]{AKL}, which was originally proved in \cite[Proposition 4.3]{Lawler continuity of tip} with a different expression.

\begin{lemma}\label{lem: boundary est 0}[Boundary estimate for $\H^*$] If $\gamma'$ is a radial SLE$_{\kappa}$ curve in $\H^*$ from $0$ to $\infty$, then for any $x' \in \R^*\sem\{0\}$ and $r>0$,
	$$\P[ \dist_{\H^*}(x',\gamma')  \leq r] \lesssim \( \frac{r}{|x'|_*} \)^{\alpha}.$$
\end{lemma}

We want to modify the boundary estimate into a more general and conformally invariant version which can be applied in more general domains. In the next lemma, we will derive an estimate involving the extremal distance between two crosscuts in the cylinder $\H^*$. The extremal distance will be determined by the domain between the crosscuts, and will be the same as the extremal distance between a representation of each of them in $\H$.

\begin{lemma}\label{lem: boundary 2}[Boundary estimate, extremal distance version]
	Let  $\gamma$ be radial SLE$_{\kappa}$ curve in a simply connected domain $D$ from a prime end $w_0$ to an interior point $z_0$. Let $\rho, \eta$ be a pair of disjoint crosscuts in $D$ 
	such that $D(\rho;\eta)$ is not a neighborhood of either $z_0$ or $w_0$. Here $w_0$ may be a prime end of $D^*(\rho;\eta)$, or determined by an endpoint of $\rho$.
	Then
	\begin{equation}\label{eq: general boundary}
	\P[ \gamma \cap \eta \neq \emptyset] \lesssim  e^{-\alpha \pi d_D(\rho,\eta)}.
	\end{equation}
\end{lemma}
\begin{proof}
	By conformal invariance of SLE, WLOG, we may assume that $D=\H^*\cup\{\infty\}$, $w_0=0$ and $z_0=\infty$. From the assumption of $\eta$ and $\rho$, we can find disjoint crosscuts $\til\eta$ and $\til\rho$ in $\H$  such that  $e^{2i}(\til\eta)=\eta$, and $e^{2i}(\til\rho)=\rho$, and $\til\rho$ disconnects $\til\eta$ from $(-\infty,0]$ and $[\pi,\infty)$ in $\H$. Let $x_1\le x_2\in[0,\pi]$ be the two endpoints of $\til\eta$. Let $\Omega$ (resp.\ $\til\Omega$) be the subdomain of $\H^*$ (resp.\ $\H$) bounded by $\eta$, $\rho$ and $\R^*$ (resp.\ $\til\eta$, $\til\rho$ and $\R$). Then $e^{2i}$ maps $\til\Omega$ conformally onto $\Omega$. First, suppose $x_1\le \pi/2$. Then $x_1=|e^{2i}(x_1)|_*$. By properties of extremal distance (cf.\ \cite{Ahlf}), we have
	$$d_{\H^*}(\rho,\eta)=d_{\Omega}(\rho,\eta)=d_{\til\Omega}(\til\rho,\til\eta)=d_{\H}(\til\rho,\til\eta)\le d_{\H}(\til\eta,(-\infty,0]).$$
	Let $r=\sup_{z\in\til\eta}\{|z-x_1|\}$. Then $\eta\in B_{\H^*}(e^{2i}(x_1),r)$. By Lemma \ref{lem: boundary est 0},
	$$\P[ \gamma \cap \eta \neq \emptyset] \le \P[\dist_{\H^*}(e^{2i}(x_1),\gamma)\le r]\lesssim \min\{1,(r/x_1)^\alpha\}.$$
	From Lemma \ref{lem: extremal distance hull},
	$$\min\{1,(r/x_1)^\alpha\}\lesssim e^{-\alpha\pi d_{\H}(\til\eta,(-\infty,0])}. $$
	Combining the above three displayed formulas, we get the desired estimate.
	
	The remaining case is $x_1\ge \pi/2$. Then $x_2\ge\pi/2$ and $\pi-x_2=|e^{2i}(x_2)|_*$. For the rest of the proof, we use the same argument as above except that we use $e^{2i}(x_2)$, $r/(\pi-x_2)$ and $[\pi,\infty)$ in place of $e^{2i}(x_1)$, $r/x_1$ and $(-\infty,0]$, respectively.
\end{proof}

For convenience, we will also need the boundary estimate in the following form.

\begin{lemma}\label{lem: final boundary est} [Boundary estimate, another version]
	Let $\gamma$ be a radial SLE$_{\kappa}$ curve in a simply connected domain $D$ from a prime end $w_0$ to an interior point $z_0$. Let $\rho$ be a crosscut in $D$ such that $D^*(\rho;z_0)$ is not a neighborhood of $w_0$ in $D$, and let $S \subset D^*(\rho;z_0)$. Let $\til{D}$ be a domain that contains $D$, and $\til{\rho}$ be a subset of $\til{D}$ that contains $\rho$. Let $\til{\eta}$ be either a Jordan curve in $\til{D}$ which intersects $\partial D$ or a crosscut in $\til{D}$. Suppose that $\til{\eta}$ disconnects $\til{\rho}$ from $S$ in $\til{D}$. Then
	$$\P[\gamma \cap S \neq \emptyset] \lesssim e^{-\alpha \pi d_{\til{D}} (\til{\rho},\til{\eta})}.$$
\end{lemma}
\begin{proof}
	By Lemma \ref{lem: first subcrosscut}, $\til{\eta}$ has a subcrosscut $\eta$ in $D$ which disconnects $S$ from $\rho$. Since $S \subset D^*(\rho)$, we have $\eta \subset D^*(\rho;z_0)$ and $S \subset D^*(\eta)$. Therefore, $D(\rho;\eta) = D^*(\rho;z_0)$ is neither a neighborhood of $z_0$ nor of $w_0$ in $D$. Using  Lemma \ref{lem: boundary 2} and the fact that $\eta$ disconnects $w_0$ from $S$, we see that
	$$\P[\gamma \cap S \neq \emptyset] \leq \P[ \gamma \cap \eta \neq \emptyset ] \lesssim e^{-\alpha \pi d_{D}(\rho,\eta)} \le e^{-\alpha \pi d_{\til{D}}(\til{\rho}, \til{\eta})},$$
	where the last inequality follows from the comparison principle for extremal length.
\end{proof}

The point estimate describes how difficult a radial SLE$_\kappa$ curve can gets close to a boundary point or interior point.
The following one-point estimate is \cite[Proposition 5.3]{AKL}.

\begin{lemma}\label{prop: int est 1}[One-point estimate for $\H^*$, conformal radius version] Let $\gamma'$ be a radial SLE$_{\kappa}$ curve in $\H^*$ from $0$ to $\infty$. If $z'\in \H^*$ with $y'=\Imm(z') \leq 1$ and $\epsilon  \leq 1/2$, then
$$  \P\[ \Upsilon_{\infty}(z') \leq \epsilon \Upsilon_0(z') \] \lesssim \( \frac{y'}{|z'|_*} \)^{\alpha} \epsilon^{2-d},$$ where $\Upsilon_t(z')$ is a half of the conformal radius seen from $z'$ at the time $t$, i.e.,  $\Upsilon_t(z')=\frac 12\crad_{H_t}(z')$, where for $0\le t\le \infty$, $H_t$ is the connected component of $\H^* \backslash \gamma'[0,t]$ that contains $z'$.
\end{lemma}

We now state and prove a one-point estimate using Euclidean distance.
\begin{lemma}\label{lem: int est 2}[One-point estimate for $\H^*$, Euclidean distance version]
Let $\gamma'$ be a radial SLE$_{\kappa}$ curve in $\H^*$ from $0$ to $\infty$. If $z'\in \H^*$ with $y'=\Imm z' \leq \ln(2)/2$, then for any $r\in(0,y')$,
$$\P[\dist_{\H^*}(z',\gamma') \leq r] \lesssim \( \frac{y'}{|z'|_*} \)^{\alpha} \( \frac{r}{y'} \)^{2-d} = \frac{P_{y'}(r)}{P_{y'}(|z'|_*)}.$$
\end{lemma}
\begin{proof}
The equality in the displayed formula follows from the definition of $P_{y'}$.
Let $z=e^{2i}(z') \in \D$ and $D_t=e^{2i}(H_t)\cup\{0\}\subset\D$ for  $t<\infty$. Let $\Upsilon_{D_t}(z)=\frac 12\crad_{D_t}(z)$. Then we have $\Upsilon_{D_t}(z) = |(e^{2i})'(z')|\Upsilon_{t}(z')=2e^{-2y'}\Upsilon_{t}(z')$  for all $t < \infty$. By Koebe's $1/4$ theorem and Lemma \ref{lem: Koebe} we have
$$\Upsilon_{t}(z') = \frac 12 e^{2y'} \Upsilon_{D_t}(z) \leq 2 \dist(z,\partial D_t)\leq 8\dist_{\H^*}(z', \partial H_t)\leq 8\dist_{\H^*}(z, \gamma'[0,t]).$$
Taking $t \to \infty$ gives $\Upsilon_{\infty}(z') \leq 8\dist_{\H^*}(z',\gamma')$. Thus, by Proposition \ref{prop: int est 1}, if $\epsilon<1/2$, then
$$\P[ \dist_{\H^*}(z',\gamma') \leq \Upsilon_0(z') \frac{\epsilon}{8} ]\leq \P[\Upsilon_{\infty}(z') \leq \epsilon \Upsilon_0(z')] \lesssim \( \frac{y'}{|z'|_*} \)^{\alpha} \epsilon^{2-d}.$$
Letting $r=\Upsilon_0(z') \frac{\epsilon}{8}$ and using that $\Upsilon_0(z')=\sinh(2y')/2\ge y'$, we then  get the desired inequality in the case $r<y'/16$. If $r\ge y'/16$, then $r\asymp y'$.
Let $x'=\Ree z'\in\R^*$. If $|x'|_*\le y'$, then $|z'|_*\asymp y'$, and so $\( \frac{y'}{|z'|_*} \)^{\alpha} \( \frac{r}{y'} \)^{2-d}\asymp 1$, and the inequality obviously holds. If $|x'|_*\le y'$, then $|z'|_*\asymp |x'|_*$. By Lemma \ref{lem: boundary est 0},
$$\P[\dist_{\H^*}(z',\gamma') \leq r] \le \P[\dist_{\H^*}(x',\gamma') \leq 17 r]\lesssim \Big(\frac{r}{|x'|_*}\Big)^\alpha\asymp \( \frac{y'}{|z'|_*} \)^{\alpha} \( \frac{r}{y'} \)^{2-d} .$$
\end{proof}

We now extend the one-point estimate to $\H^*$-domains, and  remove the assumption $y' \leq \ln(2)/2$.

\begin{lemma}[One-point estimate for $\H^*$-domains]\label{lem: one point estimate}
Suppose that $H$ is an $\H^*$ domain, and $\gamma'$ is radial SLE$_{\kappa}$ in $H$ from a prime end $w_0'$ to $\infty$. Fix $z_0' \in \lin{\H^*}$ with   $\Imm z_0'=y_0'$, and let $\pi/2>R > r >0$. Suppose that $B_{\H^*}(z_0',R) \subset H$ and that $w_0'$ is not a prime end of $B_{\H^*}(z_0',R)$. 
Then
	$$\P[\dist_{\H^*}(z_0',\gamma')\le r] \lesssim \frac{P_{y_0'}(r)}{P_{y_0'}(R)}.$$
\end{lemma}
\begin{proof}
The proof breaks down into three cases, each depending on how far from $\R^*$ is the point $z_0'$.

Case 1 (the far away case): $y_0' > R$. In this case, we have $\frac{P_{y_0'}(r)}{P_{y_0'}(R)}=(\frac r R)^{2-d}$. We first assume that $y_0' \leq \ln(2)/2$. Let $h_H:H \to \H^*$ be the canonical conformal map taking $w_0'$ to $0$ and $\infty$ to $\infty$, so that $\til{\gamma}=h_H(\gamma')$ is a radial SLE$_{\kappa}$ curve from $0$ to $\infty$ in $\H^*$. Let $\til{z}_0=h_H(z_0')$, and let $\til{y}_0=\Imm(\til{z}_0)$. Then $\til y_0\le y_0'\leq\ln(2)/2$. Lemma \ref{lem: Growth thm est} and Lemma \ref{lem: int est 2} imply that, if $r<R/7$,
$$\P[\dist_{\H^*}(z_0',\gamma')\le r]  \leq \P[\dist_{\H^*}(\til{z}_0,\til{\gamma}) \leq \til{r}] \lesssim \frac{P_{\til{y}_0}(\til{r})}{P_{\til{y}_0}(|\til{z}_0|_*)}\lesssim \frac{P_{\til{y}_0}(\til{r})}{P_{\til{y}_0}(\til{R})} \le \( \frac{\til{r}}{\til{R}} \)^{2-d} \asymp \( \frac{r}{R} \)^{2-d},$$
where $\til{r},\til{R}$ are defined as in Lemma \ref{lem: Growth thm est}, where the domain $\Omega\subset \C$ is the union of the $\pi$-periodic representative domain of $H$ in $\H$, its reflection about $\R$, and all points on $x\in\R$ such that  $e^{2i}(B(x,r)\cap \H)\subset H$ for some $r>0$, and the conformal map $\phi$ is a lift of $h_H$ under the equivalence relation from $\Omega$ into $\C$. The assumption $r<R/7$ can be removed since $\P[\dist_{\H^*}(z_0',\gamma')\le r]\le 1$.

Assume now that $y_0' > \ln(2)/2$. For each $t\ge 0$, let $H_t$ be the unbounded connected component of $H\sem \gamma'[0,t]$, and let $h_t:H_t\to H$ be the unique conformal map from $H_t$ onto $\H^*$, which satisfies $h_t(z)\to\infty$ and $\Ree(h_t(z)-z)\to 0$ as $z\to\infty$.  
Define a stopping time by $\tau=\inf\{t \geq 0 : \Imm h_t(z_0') \leq \ln(2)\}$.  By DMP of radial SLE, conditional on $\gamma'[0,\tau]$, the $\sigma$-algebra generated by $\gamma'$ before $\tau$, $\gamma'(\tau+t)$, $t\ge 0$, is a radial SLE$_\kappa$ curve in $H_\tau$ from $\gamma'(\tau)$ to $\infty$. There is a constant $R'>0$ such that $\dist_{\H^*}(\gamma[0,\tau],z_0')\ge R'$. This is true because the harmonic measure of the circle $\rho:=\pa B_{\H^*}(z_0',\dist_{\H^*}(\gamma[0,\tau],z_0'))$ in $H_\tau$ viewed from $\infty$ is the same as the harmonic measure of $h_\tau(\rho)$ in $\H^*$ viewed from $\infty$, which is bounded below by a constant since $h_\tau(\rho)$ is a connected set that touches $\R^*$ and disconnects $h_\tau(z_0')$ with $\Imm h_\tau(z_0')=\log(2)/2$ from $\infty$. Now we assume that $R< R'$.  Then $B_{\H^*}(z_0',R)\subset H_\tau$.

 Define $r_{\tau}, R_{\tau}$ as in Lemma \ref{lem: Growth thm est} at $z_0'$ with respect to $\Omega\subset\H$ being a representative domain of $B_{\H^*}(z_0',R)$ and some conformal map $\phi$ from $\Omega$ into $\H$, which is a lift of $h_\tau$. Lemma \ref{lem: Growth thm est} and Lemma \ref{lem: int est 2} imply that, if $r<R/7$,
$$\P[\dist_{\H^*}(z_{\tau}', \gamma^{\tau}) \le r|\gamma'[0,\tau]]\lesssim \( \frac{r_{\tau}}{R_{\tau}} \)^{2-d} \asymp \( \frac{r}{R} \)^{2-d}.$$
We may again remove the assumption $r<R/7$ since the lefthand side is no more than $1$. Taking expectation, we  get the inequality in Case 1 with the assumption that $R<R'$. If $R\ge R'$, then the above displayed formula holds with $R'$ in place of $R$. Since $R'\le R\le\pi/2$, and $R'$ is a constant, we get $R\asymp R'$. So the proof of Case 1 is complete.

Case 2 (the close case): $0 \le y_0' <r$. In this case, we have $\frac{P_{y_0'}(r)}{P_{y_0'}(R)}=(\frac r R)^{\alpha}$. We will use the boundary estimates to derive an upper bound in the form of $( \frac{r}{R} )^{\alpha}$. By modifying the constant slightly, we can assume that $R>4r$. Then in order to gets within distance $r$ from $z_0'$, $\gamma'$ must pass through $\rho':=\{ z' \in \H^*: |z'-\Ree(z_0')|_*=R/2 \}$ and $\eta':=\{ z' \in \H^* : |z'-\Ree(z_0')|_* =2r \}$, which are two semicircles such that $d_D(\rho',\eta')=(1/\pi)\ln(R/4r)$. Using Lemma \ref{lem: boundary 2}, we get that
$$\P[\dist_{\H^*}(z_0',\gamma')\le r] \leq \P[\gamma' \cap \eta' \neq \emptyset] \lesssim e^{-\alpha \pi d_D(\eta', \rho')} = \( \frac{2r}{R/2} \)^{\alpha}.$$

Case 3 (the middle distance case): $r \le y_0' \le R$. Let $\rho'=\{z' \in \H^*: |z'-z_0'|_*=y_0'\}$ which is a circle tangent to $\R^*$. Let $T=\inf\{t >0 : \gamma'(t) \in \rho' \}$, which is a stopping time so that $\{\dist_{\H^*}(z_0',\gamma')\le r\} \subset \{ T < \infty\}$. From Case 2, we have
$$\P[T <\infty] \lesssim \frac{P_{y_0'}(y_0')}{P_{y_0'}(R)}.$$
Define $\gamma^T(t)=\gamma'(T+t)$. By the DMP of radial SLE and Case 1, we can see that
$$\P[\dist_{\H^*}(z_0',\gamma')\le r| \gamma'[0,T], T<\infty]=\P[\dist_{\H^*}(z_0',\gamma^T)\le r | \gamma'[0,T], T<\infty] \lesssim \frac{P_{y_0'}(r)}{P_{y_0'}(y_0')}.$$
Combining these two inequalities completes the proof of Case 3.
\end{proof}

The following one-point estimate for $\D$-domains will be more useful for us.
\begin{lemma}[One-point estimate for $\D$-domains]\label{lem: one point est disc}
Let $\gamma$ be a radial SLE$_{\kappa}$ curve in a $\D$ domain $D$ from a prime end $w_0$ to $0$. Let $z_0 \in \overline{\D}$, $y_0=1-|z_0|$, and $0 < r < R < |z_0|$. Assume that $B(z_0,R)\cap\D \subset D$ and $w_0$ is not a prime end of $B(z_0,R)\cap\D$. Let $\rho=\{ z \in D: |z-z_0|=R \}$, and $\eta=\{ z \in D : |z-z_0|=r \}$. Then
$$\P[\dist(z_0, \gamma )\le r ] \lesssim \frac{P_{y_0}(r)}{P_{y_0}(R)}. $$
\end{lemma}
\begin{proof}
Since $\P[\dist(z_0, \gamma )\le r ] \le 1$, we may assume that $r<R/7$.  Let $\gamma',H,w_0',z_0'$ be the image of $\gamma,D,w_0,z_0$ under $(e^{2i})^{-1}$. So $\gamma'$ is a radial SLE$_\kappa$ curve in the $\H^*$-domain $H$ from its prime end $w_0'$ to $\infty$. Let $y_0'=\Imm z_0'=\frac 12\ln(\frac 1{|z_0|})=\frac 12\ln(\frac 1{1-y_0})$.  Since $R<|z_0|$, $(e^{2i})^{-1}$ restricted to $B(z_0,R)$ may be lifted to a conformal map $\phi$ into $\C$, i.e., $e^{2i}\circ \phi$ is identity on $B(z_0,R)$. Let $\til r$ and $\til R$ be as defined in Lemma \ref{lem: Growth thm est} with respect to $\Omega=B(z_0,R)$ and such $\phi$. Then $B_{\H^*}(z_0',\til R)\subset H$ and $w_0'$ is not a prime end of $B_{\H^*}(z_0',\til R)$.  We calculate that $\til R\le R/(2|z_0|)<1/2$. 
By Lemma \ref{lem: one point estimate}, we have
$$\P[ \dist(z_0, \gamma )\le r  ]\le \P[\dist(z_0', \gamma' )\le \til r ]\lesssim \frac{P_{y_0'}(\til r)}{P_{y_0'}(\til R)}
.$$
To finish the proof, we need to show that $\frac{P_{y_0'}( \til r)}{P_{y_0'}(  \til R)}\asymp \frac{P_{y_0}(  r)}{P_{y_0}(  R)}$. If $y_0\le 1/4$, then $y_0'=\frac 12\ln(\frac 1{1-y_0})\asymp y_0$, $|\phi'(z_0)|=1/(2|z_0|)\asymp 1$, $\til R\asymp R$, and $\til r\asymp r$. So we get the desired estimate. If $y_0\ge 1/4$, then $y_0'\ge \frac 12\ln(4/3)$. Since $R,\til R\lesssim 1\lesssim y_0,\til y_0$, we get $\frac{P_{y_0'}(  \til r)}{P_{y_0'}(  \til R)}\asymp (\frac{\til r}{\til R})^{2-d}\asymp (\frac{  r}{  R})^{2-d}\asymp \frac{P_{y_0}(  r)}{P_{y_0}(  R)}$, as desired.
\end{proof}

\section{Components of crosscuts}
Before we state the main theorem of this section, we will introduce the notation to be used. Let $(\mathcal{F}_t)$ be the right continuous filtration determined by the radial SLE curve $\gamma$. For any set $S \subset \overline{\D}$, let $\tau_S=\inf\{ t \geq 0: \gamma(t) \in S \}$. Define $D_t=\D(\gamma[0,t])$. For any stopping time $\tau$, define $\gamma^{\tau}(t)=\gamma(\tau+t)$, $t\ge 0$. Recall that, conditional on $\mathcal{F}_\tau$, $\gamma^\tau$ is a radial SLE$_\kappa$ curve in $D_\tau$.

\begin{theorem}\label{thm: ordered Green's function}
 Let $\gamma$ be a radial SLE$_{\kappa}$ curve in $\D$ from $1$ to $0$. Suppose that $z_0, z_1, \dots, z_m \in \overline{\D}\backslash \{0,1\}$. For each $z_j$, let $0 < r_j \leq R_j$, and define the circles $\ha{\xi}_j=\{ |z-z_j|= R_j \}$ and $\xi_j=\{|z-z_j|=r_j\}$. Assume that neither $0$ nor $1$ are contained in $\D^*( \ha{\xi}_j;z_j )$ for each $j$, and that $\overline{\D^*( \ha{\xi}_j )} \cap \overline{\D^*( \ha{\xi}_k )} =\emptyset$ for $j\neq k$.
Let $r_0' \in (0,r_0)$ and define $\xi_0'=\{ |z-z_0|=r_0' \}$. Define the event
$$E =\{  \tau_{\xi_0} <  \tau_{\ha{\xi}_1}  \leq \tau_{\xi_1}  < \cdots  < \tau_{\ha{\xi}_m}  \leq \tau_{\xi_m} < \tau_{\xi_0'} < \infty    \}.$$ If $y_j=1-|z_j|$, then for some constant $C$,
$$ \P[E | \mathcal{F}_{\tau_{\xi_0}}] \leq C^m \( \frac{r_0}{R_0} \)^{\alpha/4} \prod_{j=1}^m \frac{ P_{y_j}(r_j) }{P_{y_j}(R_j)}.$$
\end{theorem}

\no{\bf Remark.} The proof is similar to the proof of \cite[Theorem 3.1]{Rez Zhan upper bound}, but for completeness we include complete details.

\begin{proof}
Consider the discs $\xi$ which intersect the boundary. We know that the probability that $\gamma$ hits the points in $ \xi \cap \partial \D$ is equal to $0$, and so $\tau_{\xi}=\tau_{\xi \cap \D}$ a.s. Therefore, we can assume that each $\xi$ is either a Jordan curve or a crosscut in $\D$.  For each $j=0,1,\dots,m$, let $\tau_j=\tau_{\xi_j}$ and $\ha{\tau}_j=\tau_{\ha{\xi}_j}$, and define $\tau_{m+1}=\tau_{\xi_0'}$.

By the Domain Markov Property of SLE and Lemma \ref{lem: one point est disc}, we see that, for some constant $C$,
\begin{equation}
\P[\tau_j < \infty | \mathcal{F}_{\ha{\tau}_j}] \leq C \frac{P_{y_j}(r_j)}{P_{y_j}(R_j)}.
\end{equation}\label{eq: P of tau_j}
Combining these together gives
$$\P[E| \mathcal{F}_{\tau_0}]  \leq C^m \prod_{j=1}^m \frac{P_{y_j}(r_j)}{P_{y_j}(R_j)}.$$ If $r_0=R_0,$ then we are done. Suppose that $R_0 >r_0$.  Create a new arc $\rho=\{ z \in \D : |z-z_0|=\sqrt{R_0 r_0} \}$, so that $\rho$ is either a Jordan curve or a crosscut in $\D$ between $\xi_0$ and $\ha{\xi}_0$. Therefore, we know that
\begin{equation}\label{eq: rho extremal distances}
d_{\D}(\rho,\xi_0), d_{\D}(\rho, \ha{\xi}_0) \geq \frac{\ln(R_0/r_0)}{4\pi}.
\end{equation}

Note that $\rho$ separates $\xi_0$ from $0$. In the following argument, we will need to keep track of how the $\D$-domains $D_t$ are divided by $\rho$ at any particular time. Let $T=\inf\{ t \geq 0: \xi_0' \not\subset D_t\}$. Then if $\tau_0 \leq t <T$, $\xi_0'$ is a connected subset of $D_t$. In this case, since the starting point $1$ is outside of $\ha{\xi}_0$ and $\gamma$ intesercts $\xi_0$, it must be that $\gamma$ intersects $\rho$, and so $\rho$ intersects $\partial D_t$. By Lemma \ref{lem: first subcrosscut}, there is a first subcrosscut of $\rho$ in $D_t$, to be denoted $\rho_t$, which separates $\xi_0'$ from $0$ for each $\tau_0 \leq t <T$.

Now, we need to break the event $E$ into several cases based on the behavior of the curve $\gamma$ as it intersects the circles in the correct order. Let $I=\{ (j,j+1) : 0 \leq j \leq m \} \cup \{ (j,j): 1 \leq j \leq m \}$, and define a sequence of events $\{A_i : i \in I\}$  by
\begin{enumerate}
\item $A_{(0,1)} = \{ T > \tau_0 \} \cap \{ \D^*(\xi_1) \subset D_{\tau_0}^*(\rho_{\tau_0}) \}$

\item $A_{(j,j)} = \{ T > \tau_j \} \cap \{ \D^*(\xi_j) \subset D_{\tau_{j-1}}(\rho_{\tau_{j-1}}) \} \cap \{ \D^*(\xi_j) \subset D_{\tau_j}^*( \rho_{\tau_j} ) \}$, $1 \leq j \leq m$.

\item $A_{(j,j+1)} = \{ T > \tau_j \} \cap \{ \D^*(\xi_j) \subset D_{\tau_j} (\rho_{\tau_j}) \} \cap \{ \D^*(\xi_{j+1}) \subset D_{\tau_j}^*(\rho_{\tau_j}) \}$, $1 \leq j \leq m-1$.

\item $A_{(m,m+1)} = \{ T > \tau_m\} \cap \{ \D^*(\xi_m) \subset D_{\tau_m}(\rho_{\tau_m}) \}$.
\end{enumerate}
Observe that for each $j$, the events $A_{(j,j)}, A_{(j,j+1)}$ are $\mathcal{F}_{\tau_j}$ measurable. We claim  that
\begin{equation}\label{eq: E in Ais}
E \subset \bigcup_{i \in I} A_i.
\end{equation}
To see this, observe that $A_{(m,m+1)}$ is the event that at time $\tau_m$, $\xi_m$ lies outside of $\rho_{\tau_m}$, relative to $0$, i.e., $\xi_m$ lies in the same connected component as $0$ of $D_{\tau_m}\sem \rho_{\tau_m}$. If that does not happen, then $\xi_m$  lies inside of $\rho_{\tau_m}$ at time $\tau_m$. Suppose futher that $A_{(m,m)}$ does not happen, which is the event that at time $\tau_{m-1}$, $\xi_m$ lies outside of $\rho_{\tau_{m-1}}$, but at time $\tau_m$, $\xi_m$ lies inside of $\rho_{\tau_m}$. Then it must be that $\xi_m$ lies inside of $\rho_{\tau_{m-1}}$ at $\tau_{m-1}$. Proceeding along this way inductively proves (\ref{eq: E in Ais}). 
Now it will suffice to show that
\begin{equation}\label{eq: prob of E and Ai}
\P[E \cap A_i | \mathcal{F}_{\tau_0}] \leq C^m \( \frac{r_0}{R_0} \)^{\alpha/4} \prod_{j=1}^m \frac{P_{y_j}(r_j)}{P_{y_j}(R_j)}
\end{equation}
for each $i \in I$.
We will break it down to the four cases $i=(m,m+1), (j,j), (j,j+1),$ and  $(0,1)$. In all of these cases, we will use the convention $\gamma^T(t)=\gamma(T+t)$ for any $T,t \geq 0$.

\underline{Case $(0,1)$}. Suppose that $A_{(0,1)}$ occurs, and that $\tau_0 < \ha{\tau}_1$. See Figure \ref{A(0,1)(j,j+1)}. Then we have $\ha{\xi}_1 \subset D^*_{\tau_0}(\rho_{\tau_0})$.
Also, note that $\rho$ disconnects $\ha{\xi}_1$ from $\xi'_0$ in $\D$, and $\rho$ must intersect $\partial D_{\tau_0}$. By Lemma \ref{lem: first subcrosscut}, there is a subcrosscut $\rho'_{\tau_0}$ of $\rho$ which is first to separate $\ha{\xi}_1$ from $\xi'_0$ in the domain $D_{\tau_0}$. Since both $\ha{\xi}_1$ and $\xi'_0$ lie in $D^*_{\tau_0}(\rho_{\tau_0})$, so does $\rho'_{\tau_0}$. Note that this implies that $\rho'_{\tau_0} \neq \rho_{\tau_0}$, and that $D^*_{\tau_0}(\rho'_{\tau_0}) \subset D^*_{\tau_0}(\rho_{\tau_0})$. Since $\rho_{\tau_0}$ was defined to be the first subcrosscut of $\rho$ in $D_{\tau_0}$ that disconnects $\xi'_0$ from $0$, and since $D^*_{\tau_0}(\rho'_{\tau_0})$ is contained in the domain determined by $\rho_{\tau_0}$, it cannot be that $\rho'_{\tau_0}$ disconnects $\xi'_0$ from $0$. Therefore, we conclude that $\xi'_0 \subset D_{\tau_0}(\rho'_{\tau_0})$ and $\ha{\xi}_1 \subset D^*_{\tau_0}(\rho'_{\tau_0})$.

Observe that $\D^*(\xi_0)$ is a connected subset of $D_{\tau_0}\backslash \rho'_{\tau_0}$, and contains $\xi'_0$ and a curve which approaches $\gamma(\tau_0) \in \xi_0$. Therefore, $$D_{\tau_0}(\rho'_{\tau_0}; \gamma(\tau_0)) = D_{\tau_0}(\rho'_{\tau_0}; \xi'_0) = D_{\tau_0}(\rho'_{\tau_0}).$$ It follows that $D_{\tau_0}(\rho'_{\tau_0}; \ha{\xi}_1) = D^*_{\tau_0}(\rho'_{\tau_0})$ is not a neighborhood of $\gamma(\tau_0)=\gamma^{\tau_0}(0)$, where $\gamma^{\tau_0}$  (conditioned on $\mathcal{F}_{\tau_0}$) is a radial SLE$_{\kappa}$ curve in the $\D$-domain $D_{\tau_0}$. Since $\tau_0 < \ha{\tau}_1$, the event $\{ \ha{\tau}_1 < \infty \}$ implies that the conditioned SLE curve $\gamma^{\tau_0}$ visits $\ha{\xi}_1$. Since $\ha{\xi}_0$ disconnects $\ha{\xi}_1$ from $\rho'_{\tau_0} \subset \rho$ in $\D$, and $\ha{\xi}_0$ intersects $\partial D_{\tau_0}$, we can apply Lemma \ref{lem: final boundary est} to conclude that
$$\P[ \ha{\tau}_1 < \infty | \mathcal{F}_{\tau_0}, A_{(0,1)}, \tau_0 < \ha{\tau}_1 ] \lesssim e^{-\alpha d_{\D}(\rho, \ha{\xi}_0)} \le \( \frac{r_0}{R_0} \)^{\alpha/4}.$$
Note that the second inequality follows from (\ref{eq: rho extremal distances}). Combining the above inequality with inequality (\ref{eq: P of tau_j}) proves that inequality (\ref{eq: prob of E and Ai}) holds for the case $i=(0,1)$.

\underline{Case $(j,j+1)$}, for $1 \leq j \leq m-1$. Suppose that $A_{(j,j+1)}$ occurs, and $\tau_{j} < \ha{\tau}_{j+1}$. See Figure \ref{A(0,1)(j,j+1)}. By the same argument used in the case for $A_{(0,1)}$, we can conclude that there exists a subcrosscut of $\rho$, which we will call $\rho'_{\tau_j}$, which disconnects $\ha{\xi}_{j+1}$ from $\xi'_0$ in $D_{\tau_j}$. It follows that $D^*_{\tau_j}(\rho'_{\tau_j}) \subset D^*_{\tau_j}(\rho_{\tau_j})$, and then we can conclude that $\xi'_0 \subset D_{\tau_j}(\rho'_{\tau_j})$ and $\ha{\xi}_{j+1} \subset D^*_{\tau_j}(\rho'_{\tau_j})$. Since $\D^*(\xi_j)$ is a connected subset of $D_{\tau_j}(\rho_{\tau_j})$ and contains a curve approaching $\gamma(\tau_j) \in \xi_j,$ we can see that $D_{\tau_j}(\rho_{\tau_j}; \gamma(\tau_j))=D_{\tau_j}(\rho_{\tau_j}; \D^*(\xi_j)) = D_{\tau_j}(\rho_{\tau_j})$. Therefore, $D^*_{\tau_j}(\rho'_{\tau_j}) \subset D^*_{\tau_j}(\rho_{\tau_j})$ is not a neighborhood of $\gamma(\tau_j)=\gamma^{\tau_j}(0)$.

Since $\tau_j < \ha{\tau}_{j+1}$, the event $\{ \ha{\tau}_{j+1} < \infty \}$ implies that $\gamma^{\tau_j}$, the radial SLE$_{\kappa}$ curve in $D_{\tau_j}$ after conditioning on $\mathcal{F}_{\tau_j}$, visits $\ha{\xi}_{j+1}$. Since $\ha{\xi}_0$ disconnects $\ha{\xi}_{j+1}$ from $\rho$, and therefore from $\rho'_{\tau_j}$, in $\D$, Lemma \ref{lem: final boundary est} implies that
$$\P[\ha{\tau}_{j+1} < \infty | \mathcal{F}_{\tau_j}, A_{(j,j+1)}, \tau_j < \ha{\tau}_{j+1}] \lesssim e^{-\alpha \pi d_{\D}(\rho, \ha{\xi}_0)} \leq \( \frac{r_0}{R_0} \)^{\alpha/4}.$$
As in the above case, this proves inequality (\ref{eq: prob of E and Ai}) for the case $i=(j,j+1)$ when $1 \leq j \leq m-1$.

\begin{figure}
	\includegraphics[width=0.4\textwidth]{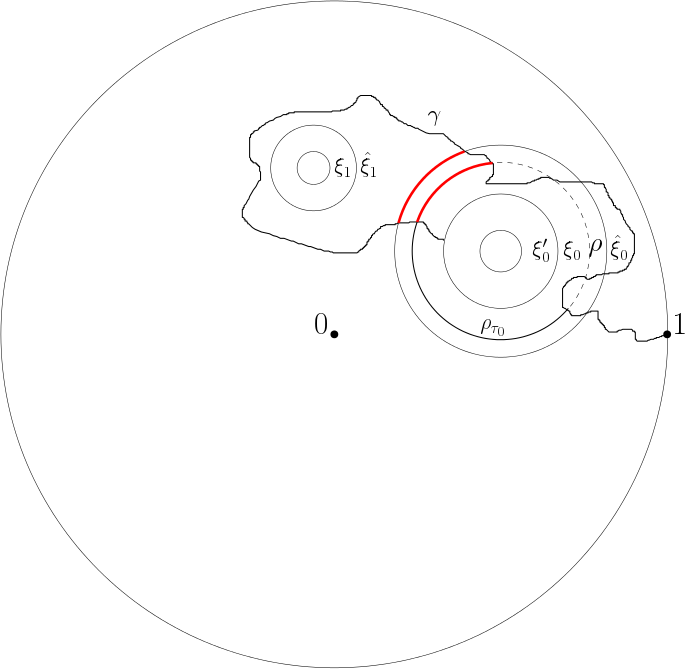}%
	\hfill
	\includegraphics[width=0.4\textwidth]{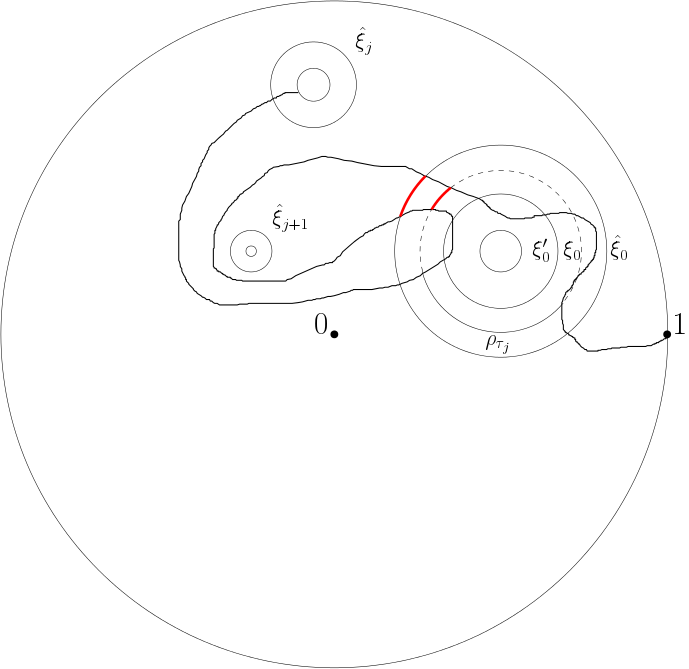}%
	\hfill
	\caption{The two pictures above illustrate the events $A_{(0,1)}$ (left) and $A_{(j,j+1)}$ (right). In each picture, the pair of arcs that contribute the boundary estimate are thickened and colored red.}
	\label{A(0,1)(j,j+1)}
\end{figure}

\underline{Case $(m,m+1)$}. Suppose that $\tau_m < \tau_{m+1}$, and that $A_{(m,m+1)}$ occurs. The set $\D^*(\xi_m)$ is connected and contained in $D_{\tau_m} \backslash \rho_{\tau_m}$, and also $\gamma(\tau_m) \in \xi_m$. This implies that $$D_{\tau_m}(\rho_{\tau_m}; \gamma(\tau_m)) = D_{\tau_m}(\rho_{\tau_m}; \D^*(\xi_m)) =D_{\tau_m}(\rho_{\tau_m}).$$ Therefore, $D^*_{\tau_m}(\rho_{\tau_m})$ is not a neighborhood of $\gamma^{\tau_m}(0)=\gamma(\tau_m)$ in $D_{\tau_m}$. Since we are assuming that $\tau_m < \tau_{m+1}$, the event $\{ \tau_{m+1} < \infty\}$ implies that the curve $\gamma^{\tau_m}$, which after conditioning on $\mathcal{F}_{\tau_m}$ is a radial SLE$_{\kappa}$ in $D_{\tau_m}$ from $\gamma(\tau_m)$ to $0$, must visit $\xi'_0 \subset D^*_{\tau_m}(\rho_{\tau_m})$. Since $\xi_0$ disconnects $\xi_0'$ from $\rho$ in $\D$ and intersects $\partial D_{\tau_m}$, we can apply Lemma \ref{lem: final boundary est} to show that
$$\P[\tau_{m+1} < \infty | \mathcal{F}_{\tau_m}, A_{(m,m+1)}, \tau_m < \tau_{m+1}] \lesssim e^{-\alpha \pi d_{\D}(\xi_0,\rho)} \leq  \( \frac{r_0}{R_0} \)^{\alpha/4}.$$

\begin{figure}
	\includegraphics[width=0.4\textwidth]{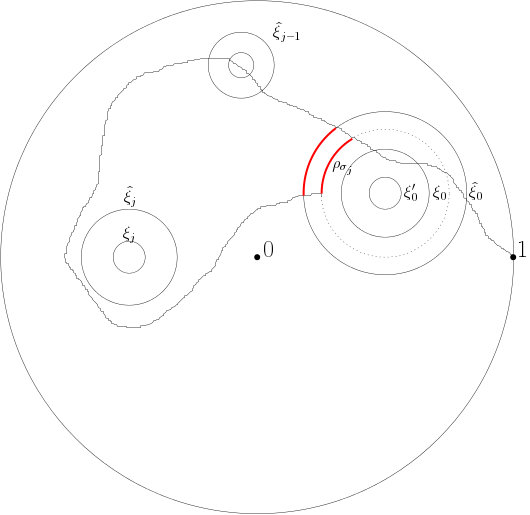}%
	\hfill
	\includegraphics[width=0.4\textwidth]{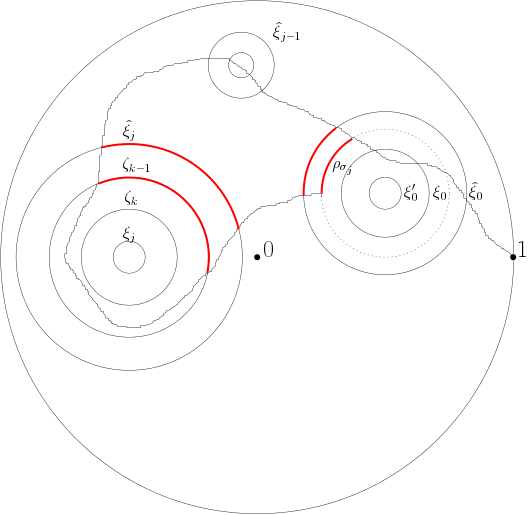}%
	\hfill
	\caption{The left pictures illustrates the subcase $F_<$ of  Case $(j,j)$. The right picture illustrates the subcase $F_k\subset F_{\ge}$ of Case $(j,j)$.  The arcs that contribute the boundary estimate in each picture are thickened and colored red.}
	\label{A(m,m+1)(j,j)}
\end{figure}

\underline{Case $(j,j)$}, for $1 \leq j \leq m$. We now prove inequality (\ref{eq: prob of E and Ai}) for $A_{(j,j)}$.  See Figure \ref{A(m,m+1)(j,j)}. Define $\sigma_j = \inf \{ t \geq \tau_{j-1}: \D^*(\xi_j) \subset D^*_t(\rho_t) \}$. This can be seen as the first time after $\tau_{j-1}$ that the SLE curve $\gamma$ hits the crosscut $\rho$ on the ``correct" side of $\ha{\xi}_j$, i.e. $\gamma$ cuts the disc off from $0$. There are a few observations to make about $\sigma_j$ which follow from \cite[Lemma 2.2]{Rez Zhan upper bound}.
First, $\sigma_j$ is an $\mathcal{F}_t$-stopping time. If $\sigma_j < \infty$, then $\D^*(\xi_j) \subset D^*_{\sigma_j}(\rho_{\sigma_j})$. If the event $A_{(j,j)}$ occurs, then so does $\{\tau_{j-1} < \sigma_j < \tau_j \}$. Finally, if $\tau_{j-1} < \sigma_j < \infty$, then it must be that $\gamma(\sigma_j)$ is an endpoint of $\rho_{\sigma_j}$. This implies that $D^*_{\sigma_j}(\rho_{\sigma_j})$ is neither a neighborhood of $\gamma(\sigma_j)$ nor of $0$.

We define two events which separate the event $A_{(j,j)}$. Define
$$F_{<}=\{\sigma_j < \ha{\tau}_j  \}, \text{ and } F_{\geq} = \{ \tau_j > \sigma_j \geq \ha{\tau}_j \}.$$ Notice that $A_{(j,j)} \subset F_< \cup F_{\geq}$, so if we can prove (\ref{eq: prob of E and Ai}) for both $F_<$ and $F_{\geq}$ instead of $A_{(j,j)}$, the case will be proven.

First, assume that $F_<$ happens. Then $\D^*(\ha{\xi}_j) \cup \ha{\xi}_j$ is a connected subset of $(\D \backslash \gamma[0,\rho_{\sigma_j}])\backslash \rho$ that contains $\D^*(\xi_j)$, and so $\ha{\xi}_j \subset D^*_{\sigma_j}(\rho_{\sigma_j} ; \D^*(\xi_j))=D^*_{\sigma_j}(\rho_{\sigma_j})$. Since $\ha{\xi}_0$ disconnects $\rho$ from $\ha{\xi}_j$ in $\D$, Lemma \ref{lem: final boundary est} and (\ref{eq: rho extremal distances}) imply that
$$\P[ \ha{\tau}_j < \infty | \mathcal{F}_{\sigma_j}, F_< ] \lesssim e^{-\alpha \pi d_{\D}(\rho, \ha{\xi}_0)} \leq \( \frac{r_0}{R_0} \)^{\alpha/4}.$$
This implies that
$$\P[ \tau_j < \infty, F_< | \mathcal{F}_{\tau_{j-1}} ] \lesssim \( \frac{r_0}{R_)} \)^{\alpha/4} \frac{ P_{y_j}(r_j) }{P_{y_j}(R_j))}.$$

Next, we assume that $F_{\geq}$ occurs, which is the more difficult case. Define $N=\lceil \ln(R_j/r_j)  \rceil \in \N$, where $\lceil \cdot \rceil$ is the ceiling function. In the event $F_{\geq}$, the SLE curve passes through the outer circle $\ha{\xi}_j$, then hits the crosscut $\rho$ before returning to hit the inner circle $\xi_j$. What we need to do is divide the annulus $\{r_j \leq |z-z_j| < R_j\}$ into $N$ subannuli until we can identify the last subcircle to be crossed before the time $\sigma_j$.

Define the circles $\zeta_k = \{ |z-z_j| = \( R_j^{N-k}r_j^j \)^{1/N} \}$ for $0 \leq k \leq N$. Note that $\zeta_0=\ha{\xi}_j$, $\zeta_N= \xi_j$, and a higher index $k$ indicates that $\zeta_k$ is more deeply nested inside $\ha{\xi}_j$. Then $F_{\geq} \subset \bigcup_{k=1}^N F_k$, where
$$F_k=\{ \tau_{\zeta_{k-1}} \leq \sigma_j < \tau_{\xi_{\zeta_k}} \}.$$ If the event $F_k$ occurs, then $\zeta_k \subset D^*_{\sigma_j}(\rho_{\sigma_j})$. This is because $\D^*(\zeta_k) \cup \zeta_k$ is a connected subset of $\( \D \backslash \gamma[0,\sigma_j] \) \backslash \rho$ which contains $\zeta_k$, $\D^*(\xi_j)$, and $D^*_{\sigma_j}(\rho_{\sigma_j})$. By Lemma \ref{lem: final boundary est} and inequality (\ref{eq: rho extremal distances}), we conclude that
\begin{equation}\label{eq: F_{geq} 1}
\P[\tau_{\zeta_k}<\infty | \mathcal{F}_{\sigma_j},F_k] \lesssim e^{-\alpha \pi d_{\D}(\rho, \zeta_{k-1})} \leq e^{-\alpha \pi( d_{\D}(\rho, \ha{\xi}_0)+ d_{\D}(\zeta_0,\zeta_{k-1}) )}
\leq \( \frac{r_0}{R_0} \)^{\alpha/4} \( \frac{r_j}{R_j} \)^{\frac{\alpha(k-1))}{2N}}.
\end{equation}
By Lemma \ref{lem: one point est disc}, we get that
$$\P[ F_k | \mathcal{F}_{\tau_{j-1}}, \tau_{j-1} < \ha{\tau}_j ] \lesssim \frac{P_{y_j}\((R_j^{N-k+1}r^{k-1})^{1/N}\)}{P_{y_j}(R_j)},$$
$$\P[ \tau_j<\infty | \mathcal{F}_{\zeta_k} , F_k] \lesssim \frac{P_{y_j}(r_j)}{P_{y_j}\( (R_j^{N-k}r_j^k)^{1/N} \)}.$$
Combining the above three inequalities and the upper bound in Lemma \ref{lem: P_y bounds},
$$ \P[ \tau_j < \infty, F_k | \mathcal{F}_{\tau_{j-1}}, \tau_{j-1} < \ha{\tau}_j ] \lesssim \( \frac{r_0}{R_0} \)^{\alpha/4} \( \frac{r_j}{R_j} \)^{ \frac{\alpha(k-1)}{2N} } \( \frac{r_j}{R_j} \)^{-\alpha/N} \frac{P_{y_j}(r_j)}{P_{y_j}(R_j)}.$$
Since $F_{\geq} \subset \bigcup_{k=1}^N F_k$, adding up the above inequality yields
$$\P[ \tau_j < \infty, F_{\geq} | \mathcal{F}_{\tau_{j-1}}, \tau_{j-1} < \ha{\tau}_j ] \lesssim \( \frac{r_0}{R_0}\)^{\alpha/4} \frac{P_{y_j} (r_j)}{P_{y_j}(R_j)} \left[  \( \frac{r_j}{R_j} \)^{-\alpha/N} \frac{ 1-( r_j/R_j )^{\alpha/2} }{1- (r_j/R_j)^{\alpha/2N}} \right]. $$
By considering the cases $R_j /r_j\le e$ and $R_j /r_j > e$ separately, we see that the quantity inside the square bracket is bounded by the constant $\frac{e^{\alpha}}{1-e^{-\alpha/4}}$.
So we get
$$\P[ \tau_j < \infty, F_{\geq} | \mathcal{F}_{\tau_{j-1}}, \tau_{j-1} < \ha{\tau}_j ] \lesssim \( \frac{r_0}{R_0}\)^{\alpha/4} \frac{P_{y_j} (r_j)}{P_{y_j}(R_j)},$$
which is the same bound as that achieved for $F_{<}$. Combining these two inequalities gives
$$\P[\tau_j<\infty, A_{(j,j)} | \mathcal{F}_{\tau_{j-1}}, \tau_{j-1}<\ha{\tau}_j] \lesssim \( \frac{r_0}{R_0} \)^{-\alpha/4} \frac{P_{y_j}(r_j)}{P_{y_j}(R_j)},$$
which completes the proof of (\ref{eq: prob of E and Ai}) in the final case, and so the theorem follows.
\end{proof}

\section{Concentric circles}
Let $\Xi$ be a family of mutually disjoint circles in $\C$ with centers in $\overline{\D} \backslash \{ 0 \}$, none of which passes through or encloses $0$ or $1$. We can define a partial order on $\Xi$  by $\xi_1 < \xi_2$ if $\xi_2$ is enclosed by $\xi_1$. Note that the larger circle has a smaller radius than the larger circle because the order is determined by visiting time.
When $\gamma$ is an SLE curve in $\D$ starting from $1$, $\xi_1 < \xi_2$ means that $\tau_{\xi_1} < \tau_{\xi_2}$, assuming $\gamma$ passes through $\xi_1$. Also, observe that circles in $\Xi$ are not necessarily contained in $\D$. In fact, we want to account for circles in $\Xi$ which have center in the boundary as well.

Further, we assume that $\Xi$ has a partition $\bigcup_{e \in \mathcal{E}} \Xi_e$ with the following properties:
\begin{enumerate}
\item For each $e \in \mathcal{E}$, the elements of $\Xi_e$ are concentric circles whose radii form a geometric sequence with a common ratio of $1/4$. For each $e$, let the common center be $z_e$. Note that elements of $\Xi_e$ are totally ordered. Let $R_e$ be the radius of the smallest circle (in the ordering on $\Xi$), and let $r_e$ be the radius of the largest circle. Then there is some integer $M \geq 0$ with $R_e=r_e 4^M$.
\item Let $A_e=\{ z \in \C : r_e \leq |z-z_e| \leq R_e \}$ denote the closed annulus containing all of the circles in $\Xi_e$. Then we assume that the collection of annuli $\{A_e\}_{e \in \mathcal{E}}$ is mutually disjoint.
\end{enumerate}
We make a couple of quick remarks about the generality of this assumption. It may be that $|\Xi_e|=1$, in which case $r_e=R_e$, and the annulus $A_e$ is just the single circle contained in $\Xi_e$. Also, if $e_1 \neq e_2$, that does not necessarily mean that $z_{e_1} \neq z_{e_2}$. 
In fact, it is possible that one annulus $A_{e_1}$ is enclosed by another annulus $A_{e_2}$.

\begin{theorem}\label{thm: concentric circles}
Let $\Xi$ be a family of circles with the properties listed above, and assume that $\gamma$ is a radial SLE$_{\kappa}$ curve in $\D$ from $1$ to $0$. For each $e \in \mathcal{E}$, let $y_e=1-|z_e|$. Then there exists a constant $C_{|\mathcal{E}|} < \infty$, which only depends on $\kappa$ and the size of the partition $|\mathcal{E}|$, so that
$$\P[ \bigcap_{\xi \in \Xi} \{ \gamma \cap \xi \neq \emptyset \}] \leq C_{|\mathcal{E}|} \prod_{e \in \mathcal{E}} \frac{P_{y_e}(r_e)}{P_{y_e}(R_e)}.$$

\end{theorem}

\no{\bf Remark.} The proof is similar to the proof of \cite[Theorem 3.2]{Rez Zhan upper bound}, but for completeness we include complete details. The strategy is to consider all possible orders $\sigma$ that the SLE curve $\gamma$ can visit all elements of $\Xi$. Under $\sigma$, $\gamma$ may pass through several elements of a family $\Xi_{e_0}$, leave and visit other $\Xi_e$'s, before returning to pass through the more inner circles in $\Xi_{e_1}$. Theorem \ref{thm: ordered Green's function} provides an estimate of the price paid by $\gamma$ in order to return to the interior circles of $\Xi_{e_0}$. This gives an estimate for the probability of $\bigcap_{\xi}\{  \gamma \cap \xi \neq \emptyset\}$ in the prescribed order $\sigma$. We then add up over all appropriate orders $\sigma$ and show that the constant only depends on $|\mathcal{E}|$.

\begin{proof}
Define $S$ to be the set of permutations $\sigma:\{1,2,\dots,|\Xi|\} \to \Xi$ such that $\xi_1<\xi_2$ implies $\sigma^{-1}(\xi_1) < \sigma^{-1}(\xi_2)$. Then $S$ is the set of viable orders in which $\gamma$ can visit the elements of $\Xi$ for the first time. For an ordering $\sigma \in S$, $\sigma(j) \in \Xi$ is the $j$-th circle visited by $\gamma$. Define the event $E^{\sigma}= \{ \tau_{\sigma(1)} < \cdots < \tau_{\sigma(|\Xi|)} < \infty\}$. Then $E:=\bigcap_{\xi \in \Xi} \{ \gamma \cap \xi \neq \emptyset \} =\bigcup_{\sigma \in S}E^{\sigma}$. What we need to do is to bound $\P[E^{\sigma}]$ for a prescribed $\sigma$.

Fix some $\sigma \in S$. Our first goal is to create a subpartition $\{ \Xi_i \}_{i \in I}$ of $\{\Xi_e\}_{e \in \mathcal{E}}$, where the elements of $\Xi_i$ receive first visits from $\gamma$ without interruption in the event $E^{\sigma}$. For each $e \in \mathcal{E}$, let $N_e=|\Xi_e|-1$, and label the elements of $\Xi_e$ by $\xi^e_0 < \cdots < \xi^e_{N_e}$. Define $J_e \subset \{ 0, 1, \dots, N_e\}$ by $J_e=\{ 1 \leq n \leq N_e : \sigma^{-1}(\xi^e_n) > \sigma^{-1}(\xi^e_{n-1})+1\} \cup \{0\}$. Then $n$ is a nonzero element of $J_e$ if, after $\gamma$ visits $\xi^e_{n-1}$, the curve visits other new circles in $\Xi$ before $\xi^e_n$. That is, there is some $\xi \in \bigcup_{e' \neq e}\Xi_e$ such that $\tau_{\xi^e_{n-1}} < \tau_{\xi} < \tau_{\xi^e_n}$. Order the elements of $J_e$ by $0=s_e(0) < s_e(1) < \cdots < s_e(M_e)$, where $M_e=|J_e|-1$ is the number of times that the progress of $\gamma$ through $\Xi_e$ is interrupted. Define $s_e(M_e+1)=N_e+1$. Using this framework, each $\Xi_e$ can be partitioned into $M_e+1$ subsets
$$\Xi_{(e,j)}=\{ \xi^e_n: s_e(j) \leq n \leq s_e(j+1)-1 \}, \text{ } 0 \leq j \leq M_e.$$ These are the elements of $\Xi_e$ which are visited without interruption. Let $I=\{ (e,j) : e \in \mathcal{E}, 0 \leq j \leq M_e \}$. Then $\{\Xi_{i}\}_{i \in I}$ is a finer partition with the desired properties.

For $i=(e,j)\in I$, let $z_i=z_{e_i}$, $y_i=1-|z_i|$, and $P_i = \frac{P_{y_i}(\rad_{\max\{ \Xi_i \}})}{P_{y_i}(\rad_{\min\{ \Xi_i \}})}$. Recall that $\min\{\Xi_i\}=\xi^{e_i}_{s_{e_i}(j)}$ and
$ \max\{ \Xi_i \}=\xi^{e_i}_{s_{e_i}(j+1)-1}$, and we use $\rad_\cdot$ to denote the radii of these two circles, respectively. By Lemma \ref{lem: one point est disc}, we can see that
\begin{equation}\label{eq: cP_i}
\P[ \tau_{\max\{ \Xi_i \}} < \infty | \mathcal{F}_{\tau_{\min\{ \Xi_i \}}} ] \leq C P_i.
\end{equation}
For $e \in \mathcal{E}$, let $P_e= P_{y_e}(r_e)/P_{y_e}(R_e)$. From Lemma \ref{lem: P_y bounds} we have
\begin{equation}\label{eq: 4alphaMe}
\prod_{j =0}^{M_e}P_{(e,j)} \leq 4^{\alpha M_e}P_e.
\end{equation}

Observe that $|I|=\sum_{e \in \mathcal{E}}(M_e+1)$ is the number of uninterrupted sequences of circles visited by $\gamma$ under $\sigma$. Then $\sigma$ induces a map $\ha{\sigma}:\{1, \dots, |I|\} \to I$ so that if $n_1 < n_2$, we have $\max\{ \sigma^{-1}( \Xi_{\ha{\sigma}(n_1)} ) \} < \min\{ \sigma^{-1}( \Xi_{\ha{\sigma}(n_2)} ) \} $, and $n_1=n_2-1$ implies $\max\{ \sigma^{-1}( \Xi_{\ha{\sigma}(n_1)} ) \}=\min\{ \sigma^{-1}( \Xi_{\ha{\sigma}(n_2)} ) \}-1$. In other words, $\ha{\sigma}$ is the order that the families $\Xi_i$ are visited. In particular, we can rewrite the event $E^{\sigma}$ as
$$E^{\sigma}= \{ \tau_{ \min \Xi_{\ha{\sigma}(1)} } < \tau_{\max \Xi_{\ha{\sigma}(1)}} < \cdots < \tau_{ \min \Xi_{\ha{\sigma}(|I|)} }<\tau_{\max \Xi_{\ha{\sigma}(|I|)}}< \infty \}.$$ We will use Theorem \ref{thm: ordered Green's function} to estimate the probability of this event.

Fix some $e_0 \in \mathcal{E}$, and let $n_j=\ha{\sigma}^{-1}((e_0,j))$ for $0 \leq j \leq M_{e_0}$. In the event $E^{\sigma}$, the family $\Xi_{(e_0,j)}$
is the $n_j$-th family to be visited by the curve. For $0 \leq j \leq M_e-1$, $n_{j+1} \geq n_j+2$ since at least one other family not contained in $\Xi_{e_0}$ must be hit by the curve between them. Fix $0 \leq j \leq M_{e_0}-1$, and let $m=n_{j+1}-n_j-1 \geq 1$. We are going to apply Theorem \ref{thm: ordered Green's function} to the family
\begin{itemize}
\item $\ha{\xi}_0=\min \Xi_{e_0}, \xi_0=\max \Xi_{(e_0,j)}=\max \Xi_{\ha{\sigma}(n_j)}$, and $\xi'_0=\min \Xi_{(e_0,j+1)}=\min \Xi_{\ha{\sigma}(n_{j+1})}$
\item $\ha{\xi}_k= \min \Xi_{\ha{\sigma}(n_j+k)}$, and $\xi_k=\max \Xi_{\ha{\sigma}(n_j+k)}$ for $1 \leq k \leq m$.
\end{itemize}
In plain words, $m$ is the number of other families $\{\Xi_i\}$ first visited by $\gamma$ between first visits of the $j$-th level and the $(j+1)$-th level of the family $\Xi_{e_0}$. The curves $\xi_k, \ha{\xi}_k$ are the $k$-th family first visited before returning to $\Xi_{(e_0, j+1)}$.

Define an event by
$$E^{\sigma}_{[ \max \Xi_{\ha{\sigma}(n_j)}, \min \Xi_{\ha{\sigma}(n_{j+1})} ]}=\{ \tau_{\xi_0} < \tau_{\ha{\xi}_{1}} < \tau_{\xi_1} < \cdots < \tau_{\xi_m} < \tau_{\xi'_0}\} \in \mathcal{F}_{\tau_{ \min\Xi_{\ha{\sigma}(n_{j+1})} }}.$$
Theorem \ref{thm: ordered Green's function} implies that
$$\P[ E^{\sigma}_{[ \max \Xi_{\ha{\sigma}(n_j)}, \min \Xi_{\ha{\sigma}(n_{j+1})} ]} | \mathcal{F}_{\max \Xi_{\ha{\sigma}(n_j)}}] \leq C^m 4^{-\frac{\alpha}{4} (s_{e_0}(j+1)-1)} \prod_{n=n_{j}+1}^{n_{j+1}-1}P_{\ha{\sigma}(n)}.$$
Varying $j=0,1, \dots, M_{e_0}-1$ and using inequality (\ref{eq: cP_i}), we see that

$$\P[E^{\sigma}] \leq C^{|I|} 4^{-(\alpha/4) \sum_{j=1}^{M_{e_0}} s_{e_0}(j)-1 } \prod_{i \in I} P_i. $$
If we use inequality (\ref{eq: 4alphaMe}) and $|I|=\sum_e M_e+1$, we can deduce that the right hand side is bounded above by
$$C^{|\mathcal{E}|}C^{\sum_{e \in \mathcal{E}}M_e}4^{-(\alpha/4) \sum_{j=1}^{M_{e_0}} s_{e_0}(j) } \prod_{e \in \mathcal{E}}P_e.$$
Recall that this estimate was based on a fixed $e_0 \in \mathcal{E}$, but the left hand side does not depend on this choice. Taking the  geometric average with respect to $e_0 \in \mathcal{E}$, we can see that
$$\P[E^{\sigma}] \leq C^{|\mathcal{E}|} C^{\sum_{e \in \mathcal{E}}M_e} 4^{-\frac{\alpha}{4|\mathcal{E}|  } \sum_{e \in \mathcal{E}} \sum_{j=1}^{M_e}s_e(j) }  \prod_{e \in \mathcal{E}}P_e.$$

Note that the finer partition $I$ and associated terms $M_e, s_e(j)$ are dependent on our initial choice of order $\sigma$. Using the fact that $E=\cup E^{\sigma}$ and the above inequality, we get
\begin{equation}\label{eq: final circle inequality}
\P[ E ] \leq C^{|\mathcal{E}|} \(\sum_{\( M_e; (s_e(j))_{j=1}^{M_e} \)_{e \in \mathcal{E}}} |S_{(M_e,(s_e(j)))}| C^{\sum_{e \in \mathcal{E}}M_e }4^{-\frac{\alpha}{4|\mathcal{E}|}\sum_{e \in \mathcal{E}}\sum_{j=1}^{M_e}s_e(j) }\)  \prod_{e \in \mathcal{E}}P_e,
\end{equation}
 where
$$S_{(M_e,s_e(j)))} =\{ \sigma \in S : M_e^{\sigma}=M_e, s^{\sigma}_e(j) =s_e(j), \text{ for } 0 \leq j \leq M_e, e \in \mathcal{E} \}$$
and the first sum in inequality (\ref{eq: final circle inequality}) is over all possible $\( M_e; (s_e(j))_{j=1}^{M_e} \)_{e \in \mathcal{E}}$. That is, for each $e$, all possible $M_e \geq 0$ and possible orderings $0=s_e(0) < s_e(1)< \cdots < s_e(M_e) \leq N_e$. Recall that $N_e=|\Xi_e|-1$ is fixed with the initial partition $\mathcal{E}$. To finish the proof of the theorem, it suffices to show that the large term in the parenthesis in (\ref{eq: final circle inequality}) can be bounded above by some finite constant depending only on $|\mathcal{E}|$ and $\kappa.$

We claim that
\begin{equation}\label{eq: number of permutations}
|S(M_e, s_e(j))|\leq |\mathcal{E}|^{\sum_{e \in \mathcal{E}}M_e+1 }.
\end{equation}
Notice that the pair $(M_e, s_e(j))$ completely determines the sub-partition $\{\Xi_i\}_{i \in I^{\sigma}}$. Suppose $\{\Xi_i\}_{i \in I^{\sigma}}$ is given. Then the ordering $\sigma$ can be recovered from the induced ordering $\ha{\sigma}: \{1, \dots, |I^{\sigma}|\} \to I^{\sigma}$. This is because the order that circles in any given $\Xi_i$ are visited is predetermined by the ordering on $\Xi$, and so moving from $\ha{\sigma}$ to $\sigma$ requires no new information. Next, we claim that $\ha{\sigma}$ is determined by knowing $e_{\ha{\sigma}(n)}$, for each $1 \leq n \leq |I^{\sigma}|$. If $e_{\ha{\sigma}(n)}=e_0$, then $\ha{\sigma}(n)=(e_0,j_0)$ where $j_0$ can be determined by $j_0=\min\{ 0 \leq j \leq M_{e_0}: (e_0,j): (e_0,j) \notin \ha{\sigma}(m), m<n \}$. There are $|I^{\sigma}|=\sum_{e \in \mathcal{E}} (M_e+1)$ terms to determine for $e_{\ha{\sigma}(n)}$, each of which has at most $|\mathcal{E}|$ possibilities, which proves (\ref{eq: number of permutations}). Using this to estimate the constant in (\ref{eq: final circle inequality}), we see that
$$\sum_{\( M_e; (s_e(j))_{j=1}^{M_e} \)_{e \in \mathcal{E}}} |S_{(M_e,(s_e(j)))}| C^{\sum_{e \in \mathcal{E}}M_e }4^{-\frac{\alpha}{4|\mathcal{E}|}\sum_{e \in \mathcal{E}}\sum_{j=1}^{M_e}s_e(j) }$$
$$ \leq \sum_{\( M_e; (s_e(j))_{j=1}^{M_e} \)_{e \in \mathcal{E}}} |\mathcal{E}|^{\sum_{e \in \mathcal{E}} (M_e+1)} C^{\sum_{e \in \mathcal{E}}M_e }4^{-\frac{\alpha}{4|\mathcal{E}|}\sum_{e \in \mathcal{E}}\sum_{j=1}^{M_e}s_e(j) }$$
$$=|\mathcal{E}|^{|\mathcal{E}|} \sum_{\( M_e; (s_e(j))_{j=1}^{M_e} \)_{e \in \mathcal{E}}} \prod_{e \in \mathcal{E}}  ( C |\mathcal{E}| )^{M_e} 4^{-\frac{\alpha}{4|\mathcal{E}|} \sum_{j=1}^{M_e}s_e(j)}$$
$$=|\mathcal{E}|^{|\mathcal{E}|} \prod_{e \in \mathcal{E}} \sum_{M_e=1}^{N_e} \( C |\mathcal{E}| \)^{M_e} \sum_{0 =s_e(0) < \cdots < s_e(M_e) \leq N_e} 4^{-\frac{\alpha}{4|\mathcal{E}|} \sum_{j=1}^{M_e}s_e(j)}
$$
$$ \leq |\mathcal{E}|^{|\mathcal{E}|} \prod_{e \in \mathcal{E}} \sum_{M=0}^{\infty} \( C|\mathcal{E}| \)^M \sum_{s(1)=1}^{\infty} \cdots \sum_{s(M)=M}^{\infty} 4^{-\frac{\alpha}{4|\mathcal{E}|} \sum_{j=1}^M s(j)}$$
$$\leq |\mathcal{E}|^{|\mathcal{E}|} \prod_{ e \in \mathcal{E}} \sum_{M=0}^{\infty} \(C |\mathcal{E}| \)^M \prod_{j=1}^M \sum_{s(j)=j}^{\infty} 4^{-\frac{\alpha}{4|\mathcal{E}|}s(j)} =|\mathcal{E}|^{|\mathcal{E}|} \( \sum_{M=0}^{\infty} \( \frac{C|\mathcal{E}|}{1-4^{-\frac{\alpha}{4|\mathcal{E}|}}} \)^M 4^{-\frac{\alpha}{8|\mathcal{E}|}M(M+1)} \)^{|\mathcal{E}|}.$$
The last equality comes from observing that the terms in the product no longer depended on $e$ and evaluating the geometric series. Note that this bound only depends on $|\mathcal{E}|$ and $\kappa$. It is finite because inside the exponent it has the form $\sum a^nb^{-n^2} < \infty$ for some $b>1$.
\end{proof}

\section{Main theorems}
The proof of Theorem \ref{thm: multi point estimate radial} is similar to that of \cite[Theorem 1.1]{Rez Zhan upper bound}.
The strategy  is to construct a family of circles $\Xi$ and a partition $\{\Xi_{e}\}_{e \in \mathcal{E}}$ satisfying the hypothesis of Theorem \ref{thm: concentric circles} from the discs $\{ |z-z_j| \leq r_j \}$ and the distances $l_j$. The constant given will depend on the size of the partition $\mathcal{E}$, but then it can be shown that that $|\mathcal{E}|$ can be bounded above in a way which depends only on the number of points $n$.

\begin{proof}[Proof of Theorem \ref{thm: multi point estimate radial}]
We can assume without loss of generality that for each $j=1, \dots, n$, the radius $r_j$ satisfies $r_j=l_j/4^{h_j}$ for some integer $h_j \geq 1$. This is because the ratio must satisfy $4^{-h_j-1} \leq r_j/l_j \leq 4^{-h_j}$ for some integer $h_j$, and any increase of $r_j$ by at most a factor of $4$ only affects the constant for each term $j$. 

\textbf{Constructing $\Xi$:} for each $1 \leq j \leq n$, construct a sequence of circles by
$$\xi^s_j = \{ |z-z_j| =\frac{l_j}{4^s} \}, \text{ for } 1 \leq s \leq h_j.$$
The family of circles $\{\xi^s_j\}_{j,s}$ may not be disjoint, so we may have to remove some circles. For any fixed $k \leq n$, let $D_k=\{ |z-z_k| \leq l_k/4 \}$, which is a closed disc containing all of the circles centered at $z_k$. For $j < k \leq n$, define $I_{j,k} =\{ \xi_j^s : 1 \leq s \leq h_j, \text{ and } \xi^s_j \cap D_k \neq \emptyset  \}$, 
and $$\Xi=\{ \xi^s_j : 1 \leq j \leq n, 1 \leq s \leq h_j \} \backslash \bigcup_{1 \leq j < k \leq n}I_{j,k}.$$ Then $\Xi$ is composed of mutually disjoint circles by construction. If $\dist(\gamma, z_j) \leq r_j$, then $\gamma$ intersects each $\xi^s_j$ for $1 \leq s \leq h_j$, and therefore
\begin{equation}\label{eq: use concentric circles}
\P[ \bigcap_{j=1}^n \{ \dist(z_j,\gamma) \leq r_j\} ] \leq \P[ \bigcap_{\xi \in \Xi} \{ \gamma \cap \xi \neq \emptyset \} ].
\end{equation}
In order to apply Theorem \ref{thm: concentric circles}, we need to partition $\Xi$.

\textbf{Partitioning $\Xi$:} The family of circles already has a partition $\{\Xi_j\}_{j=1}^n$, where $\Xi_j$ is the set of circles $\xi \in \Xi$ with center $z_j$, but this may not be sufficiently fine to satisfy Theorem \ref{thm: concentric circles}. For example, it may be that there are circles centered at $z_3$ which lie between circles centered at $z_2$. To make a finer partition, for each $\Xi_j$, we will construct a graph $G_j$ whose connected components are the uninterrupted circles. Being more precise, the vertex set of $G_j$ is $\Xi_j$, and $\xi_1,\xi_2 \in \Xi_j$ are connected by an edge if they are concentric circles with ratio of radii being $4$,
and the open annulus between $\xi_1$ and $\xi_2$ contains no other circles from $\Xi$. Let $\mathcal{E}_j$ denote the set of connected components in $\Xi_j$. Then each $\Xi_j$ can be partitioned into $\{\Xi_e\}_{e \in \mathcal{E}_j}$, where $\Xi_e$ is the vertex set of $e \in \mathcal{E}_j$. The circles in $\Xi_e$, for $e \in \mathcal{E}_j$, are all concentric circles with center $z_j$ whose radii form a geometric sequence with common ratio $1/4$, and the closed annuli $A_e$ are mutually disjoint.

By construction, we see that for any $j < k$ and $ e \in \mathcal{E}_j$, the annulus $A_e$ does not intersect $D_k$, which contains every $A_e$ for $e \in \mathcal{E}_k$. Thus, $\mathcal{E}:=\bigcup_{j=1}^{n} \mathcal{E}_j$ satisfies the hypothesis of Theorem \ref{thm: concentric circles}, and so we can bound (\ref{eq: use concentric circles}) by
\begin{equation}\label{eq: use concentric circles2}
\P[ \bigcap_{\xi \in \mathcal{E}} \{ \gamma \cap \xi \neq \emptyset \} ] \leq C_{|\mathcal{E}|} \prod_{j=1}^n \prod_{e \in \mathcal{E}_j} \frac{P_{y_j}(r_e)}{P_{y_j}(R_e)},
\end{equation}
where $r_e=\rad (\max\{ \xi \in \mathcal{E}_e \})$ and $R_e =\rad (\min\{ \xi \in \mathcal{E}_e \})$, where again the ordering for the maximum and minimum is in terms of crossing time.  It suffices now to show that $|\mathcal{E}|$ is comparable to some value depending on $n$, and that $\prod_{e \in \mathcal{E}_j} \frac{P_{y_j}(r_e)}{P_{y_j}(R_e)}$ is comparable to $\frac{P_{y_j}(r_j)}{P_{y_j}(R_j)}$, where $R_j=l_j/4$.

\textbf{Bounding $|\mathcal{E}|$:} First, we observe a useful fact. If $1 \leq j < k \leq n$, from $l_k \leq |z_j-z_k|$, we get
\begin{equation}\label{eq: fact for |E|}
\frac{\max_{z \in D_k}\{ |z-z_j| \}}{\min_{z \in D_k} \{ |z-z_j| \}}=\frac{|z_j-z_k|+l_k/4}{|z_j-z_k|-l_k/4} \leq \frac{5}{3}<4.
\end{equation}
This inequality has the following two consequences:
\begin{itemize}
\item[a)] If $1 \leq j < k \leq n$, then $|I_{j,k}| \leq 1$.
\item[b)] If $1 \leq j < k \leq n$, then $\bigcup_{\xi \in \Xi_k} \xi \subset D_k$ intersects at most 2 annuli $\{ \frac{l_j}{4^r} \leq |z-z_j| \leq \frac{l_j}{4^{r-1}} \}$.
\end{itemize}

These consequences can be used to bound $|\mathcal{E}_j|$. We may obtain $G$ by removing vertices and edges from a path graph $\ha G$, whose vertex set is $\{\xi_j:1\le j\le h_j\}$, and two vertices are connected by an edge iff the bigger radius is $4$ times the smaller one. Every edge $e$ of $\ha G$ determines an annulus, denoted by $A_e$. The vertices removed are the elements in $I_{j,k}$, $k > j$; and the edges removed are those
$e$ such that $A_e$ intersects some $\xi\in \Xi_k$ with $ k > j$. Thus, the
total number of vertices or edges removed is not bigger than $3(n-j)$. So we get $|\mathcal{E}_j| \leq 1 + 3(n-j)$. Summing over $j$ yields
$$|\mathcal{E}| = \sum_{j=1}^n |\mathcal{E}_j| \leq \sum_{j=1}^n 1+3(n-j) = n+ \frac{3n(n-1)}{2}.$$
This completes the proof that $C_{|\mathcal{E}|}\le C_n < \infty$ for some $C_n$ depending only on $n$ and $\kappa$.

\textbf{Final estimate:} To finish the proof, we need to show that
$$\prod_{e \in \mathcal{E}_j} \frac{P_{y_j}(r_e)}{P_{y_j}(R_e)} \leq C_n \frac{P_{y_j}(r_j)}{P_{y_j}(R_j)},\quad 1\le j\le n.$$ We introduce some new notation. For any annulus $A=\{z:  r \leq |z-z_0| \leq R \}$ with  $z_0 \in \overline{\D}$, let $P(A)=P_{y_0}(r)/P_{y_0}(R)$, where $y_0=1-|z_0|$. For $1 \leq s \leq h_j$, let $A_{j,s} =\{ l_j/4^{s} \leq | z-z_j| \leq l_j/4^{s-1} \}$, and let $S_j=\{ s \in \{ 1, \dots, h_j \} : A_{j,s} \subset \bigcup_{e \in \Xi_j}A_e \}$. Using this new notation, what we want to show is that
$$\prod_{e \in \mathcal{E}_j} \frac{P_{y_j}(r_e)}{P_{y_j}(R_e)} = \prod_{s \in S_j}P(A_{j,s}) \leq C_n\frac{P_{y_j}(r_j)}{P_{y_j}(R_j)} = C_n\prod_{s=1}^{h_j}P(A_{j,s}).$$

By the estimate in Lemma \ref{lem: P_y bounds}, for each $e \in \mathcal{E}_j$, we have $P_{y_j}(r_e)/P_{y_j}(R_e) \geq (r_e/R_e)^{\alpha} = 4^{-\alpha}$. Let $S_j^c= \{ 1,\dots,h_j \} \backslash S_j$. Then
$$\prod_{s \in S_j} P(A_{j,s}) =\frac{ \prod_{s=1}^{h_j} P(A_{j,s}) }{\prod_{s \in S_j^c} P(A_{j,s})} \leq \frac{P_{y_j}(r_j)}{P_{y_j}(R_j)} 4^{\alpha |S_j^c|}.$$ We need to estimate $|S_j^c|$. If $s \in S_j^c$, then either $s=1$ or there is some $k >j$ so that $D_k \cap A_{j,s} \neq \emptyset$. By consequence $b)$ of (\ref{eq: fact for |E|}), for each $k>j$, this happens at most twice for each $k >j$. Therefore, $|S_j^c| \leq 1+ \sum_{k>j} 2=1+2(n-j)$. By equation (\ref{eq: use concentric circles2}),
$$\P[\bigcap_{\xi \in \Xi} \{\gamma \cap \xi \neq \emptyset\}] \leq C_n\prod_{j=1}^n \frac{P_{y_j}(r_j)}{P_{y_j}(R_j)} 4^{\alpha(1+2(n-j))}= C'_n  \prod_{j=1}^n \frac{P_{y_j}(r_j)}{P_{y_j}(R_j)}.$$
\end{proof}


\begin{proof}[Proof of Theorem \ref{thm: multi point estimate whole-plane}]
Define $R=\max\{|z_j |:1\leq j \leq n\}$. 
Fix a big number $N>R$.
Let $T=\inf\{t>-\infty: |\gamma^*(t)|=N\}$, which is a finite stopping time. Let $D_T$ be the connected component of $\C\sem \gamma^*(-\infty,T]$ that contains $0$. Let $\phi$ be the conformal map from $D_T$ onto $\D$ which fixes $0$ and sends $\gamma^*(T)$ to $1$. By DMP of whole-plane SLE, conditioned on $\gamma^*(-\infty,T]$, the path $\gamma:=\phi(\gamma^*(T+\cdot))$ is a radial SLE$_{\kappa}$ trace in $\D$ from $1$ to $0$.

For $1\le k\le n$, let $z_k^T=\phi(z_k)\in\D$, $y_k^T=1-|z_k^T|$, $l^T_k=\min\{|z^T_k|,|z^T_k-1|,|z^T_k-z^T_1|,\dots,|z^T_k-z^T_{k-1}\}$, and $r_k^T=\max\{ |\phi(z)-\phi(z_k)| : |z-z_k|=r_k \}$. By Theorem \ref{thm: multi point estimate radial},
$$\P[\dist(z_k,\gamma^*)\le r_k,1\le k\le n|\mathcal{F}_T]\le \P[\dist(z_k^T,\gamma)\le r_k^T,1\le k\le n]\le C_n \prod_{k=1}^n \frac{P_{y_k^T}(r_k^T\wedge l_k^T)}{P_{y_k^T}( l_k^T)}.$$
By Koebe's $1/4$ theorem and distortion theorem, as $N\to\infty$, $|\phi'(0)|\to 0$, $y_k^T\to 1$, $r^T_k/(|\phi'(0)|r_k)\to 1$ and $l^T_k/ (|\phi'(0)|l_k)\to 1$, which implies that
$\prod_{k=1}^n \frac{P_{y_k^T}(r_k^T\wedge l_k^T)}{P_{y_k^T}( l_k^T)}\to \prod_{k=1}^n \big(\frac{r_k\wedge l_k}{l_k}\big)^{2-d}$.
The proof is complete by taking expectation and letting $N\to\infty$.
\end{proof}

\begin{proof}[Proof of Theorem \ref{thm: Minkowski content moments}]
First, we prove the theorem for radial SLE. Recall that we defined Cont$_d(E;r)=r^{d-2}\Area\{z \in \C : \dist(z,E) < r\}$ for any $r>0$. Fixing $r>0$, we have
$$\E[\Cont_d(\gamma;r)^n]=r^{n(d-2)}\E\[ \( \int_{\D} \I_{\{ z: \dist(z,\gamma) < r \}}(z)dA(z) \)^n \]$$
$$=\int_{\D^n} r^{n(d-2)} \P\[ \dist(z_1,\gamma)<r, \dots, \dist(z_n,\gamma)<r \] dA(z_1)\dots dA(z_n).$$ The last equality holds by Fubini's theorem. By Theorem \ref{thm: multi point estimate radial}, if $0<r<l_1,\dots,l_n$, then
$$r^{n(d-2)} \P[ \dist(z_k,\gamma) <r, \forall k=1,\dots,n ] \leq r^{n(d-2)}C_n \prod_{k=1}^n \frac{P_{y_k}(r  )}{P_{y_k}(l_k)}\leq C_n \prod_{j=1}^n l_j^{d-2}.$$

Thus, if $f(z_1,\dots,z_n)= \prod_{k=1}^n \min\{ |z_k|,|z_k-1|, |z_k-z_1|, \dots, |z_k-z_{k-1}| \}^{d-2}$ and $0<r<l_1,\dots,l_n$, then
$$\E[ \Cont_d(\gamma ;r)^n ] \leq C_n \int_{\D^n}f(z_1,\dots,z_n) dA(z_1) \dots dA(z_n).$$
By Fatou's Lemma,
$$\E[{\underline{\Cont}}_d(\gamma)^n ] \leq \liminf_{r \to 0} \E[\Cont_d(\gamma;r)^n] \leq C_n \int_{\D^n} f(z_1,\dots,z_n) dA(z_1)\dots dA(z_n).$$
If we can show that $f$ is integrable, then we are done.

Fix $k =1, \dots, n$, and let $z_1, \dots z_{k-1} \in \D$ be arbitrary. Let $z_{-1}=0$ and $z_0=1$. For any $k$,
$$\int_{\D} \min\{ |z_k|, |z_k-1|, |z_k-z_1|, \dots, |z_k-z_{k-1}| \}^{d-2}dA(z_k) \leq \sum_{j=-1}^{k-1} \int_{\D} |z_k-z_j|^{d-2}dA(z_k).$$ Note that for each $j$, $\D-z_j \subset 2\D,$ and so the righthand side of the above inequality is
$$\leq (k+1) \int_{2 \D} |z|^{d-2}dA(z)=(k+1)\pi \int_0^2 r^{d-1}dr < \infty.$$
By Fubini's theorem, it follows that $\int_{\D^n} f dA \dots dA < \infty$, and so we are done in the radial case.

The proof for whole-plane SLE$_{\kappa}$ follows similarly. We start off by writing
$$\E[ \Cont_d(\gamma^* \cap D);r)^n]= \int_{D^n} r^{n(d-2)}\P[\dist(z_1, \gamma^*)<r, \dots, \dist(z_n,\gamma^*)<r]dA(z_1)\dots dA(z_n).$$
The function $f(z_1,\dots,z_n)$ in this case is defined by $$f(z_1,\dots,z_n) = \prod_{k=1}^n \min\{ |z_k|, |z_k-z_1|, \dots, |z_k-z_{k-1}| \}^{d-2}.$$ Then
$$\E[\underline{\Cont}_d(\gamma^* \cap D)^n] \leq \int_{D^n}f(z_1,\dots, z_n) dA(z_1)\dots dA(z_n).$$
If $D \subset M\D$ for $M<\infty$, then $D-z_j \subset 2M \D$ for each $z_j \in D$, so we can perform the same bound as in the radial case, except integrating over $2M\D$ rather than $2\D$.
\end{proof}

\end{document}